\documentclass[10pt,reqno,a4paper]{amsart}

\let\oldtocsection=\tocsection

\let\oldtocsubsection=\tocsubsection

\let\oldtocsubsubsection=\tocsubsubsection

\renewcommand{\tocsection}[2]{\hspace{0em}\oldtocsection{#1}{#2}}
\renewcommand{\tocsubsection}[2]{\hspace{2em}\oldtocsubsection{#1}{#2}}
\renewcommand{\tocsubsubsection}[2]{\hspace{2em}\oldtocsubsubsection{#1}{#2}}
\usepackage{amsmath,amsthm,amsfonts,amssymb}

\usepackage{marginnote}

\usepackage{fancyhdr}

\usepackage{graphicx}

\usepackage{pdfpages}

\usepackage{tikz}
\usetikzlibrary{matrix}
\usetikzlibrary{cd}

\usepackage{xcolor}

\numberwithin{figure}{section}
\numberwithin{equation}{section}

\newtheorem{thm}{Theorem}[section]

\newtheorem{defn}[thm]{Definition}

\newtheorem{lmm}[thm]{Lemma}

\newtheorem{prp}[thm]{Proposition}

\newtheorem{remark}[thm]{Remark}

\newcommand{\tei}{Teichm\"uller}
\newcommand{\qc}{quasiconformal}

\newcommand{\idt}{of $\id$-type}
\newcommand{\nd}{\mathcal{N}_d}

\newcommand{\id}{\operatorname{id}}
\newcommand{\Id}{\operatorname{Id}}
\newcommand{\SV}{\operatorname{SV}}

\newcommand{\ord}{\operatorname{ord}}

\renewcommand{\Re}{\operatorname{Re\,}}
\renewcommand{\Im}{\operatorname{Im\,}}

\newcommand{\abs}[1]{\left| #1 \right|}

\newcounter{reminder}

\graphicspath{ {./images/}}

\title{Infinite-dimensional Thurston theory\\and transcendental dynamics II:\\classification of entire functions\\with escaping singular orbits}
\author{Konstantin Bogdanov}
\address{Aix-Marseille Universit\'e, Centrale Marseille, Institut de Math\'ematiques de Marseille, UMR7373, 39 Rue F. Joliot Curie 13453, Marseille, France}

\begin{document}

\begin{abstract}
	We classify transcendental entire functions that are compositions of a polynomial and the exponential for which all singular values (i.e., either asymptotic or critical values) escape on disjoint rays. The construction involves an iteration procedure on an \emph{infinite-dimensional} Teichm\"uller space, analogously to classical Thurston's Topological Characterization of Rational Functions, but for an \emph{infinite} set of marked points.
	
	We concentrate on the case when different escaping singular orbits do not come arbitrarily close to each other---this allows to avoid excessive technicality of constructions (almost with no restriction of generality). As in the proof of Thurston's Characterization theorem, we define a map $\sigma$ acting on a specially chosen \tei\ space and look for fixed points of $\sigma$. But unlike in the classical theorem, the \tei\ space is no longer finite-dimensional, which leads to a completely different approach. The core of this article is the construction of a compact $\sigma$-invariant subset in the \tei\ space. Then from strict contraction of $\sigma$ on this subspace follows the existence of a fixed point of $\sigma$.
	
\end{abstract}

\maketitle
	
\tableofcontents

\addtocontents{toc}{\protect\setcounter{tocdepth}{1}}

\section{Introduction}

An important notion in the study of the dynamical behavior of transcendental entire functions, as in the polynomial dynamics, is the \emph{escaping set}. For a transcendental entire function $f$, its \emph{escaping set} is defined as
$$I(f)=\{z\in\mathbb{C}: f^n(z)\to\infty\text{ as }n\to\infty\}.$$
It is proved in \cite{SZ-Escaping} for the exponential family $\{e^z+\kappa: \kappa\in\mathbb{C}\}$ and in \cite{RRRS} for more general families of functions (of bounded type and finite order) that the escaping set of every function in the family can be described via \emph{dynamic rays} and their endpoints. Roughly speaking, this means that every escaping point of $f$ (or some of its forward images) can be joined to $\infty$ by a unique injective curve (a part of the dynamic ray which is the maximal such curve) contained in $I(f)$ and so that $f$ maps it into another such curve, that is, forward images of a dynamic ray is another one. About the points that belong to a dynamic ray we say that they \emph{escape on rays}. This is the general mode of escape in such families of functions.

\begin{remark}
	However, in some cases the endpoints of rays form a ``bigger'' set than the union of all rays without endpoints, in the sense of the Hausdorff dimension \cite{Karpinska,DS-Duke,DS-Monthly}.
\end{remark}

In this article we discuss only the escape on rays and avoid the escape on endpoints because they might have a slower and less predictable orbit \cite{LasseSlowEscape}, which would have lead to a much more complicated analysis of the arising \tei\ space.

We focus on entire functions of the form $p\circ\exp$ where $p$ is a polynomial. Moreover, it is assumed that $p$ is monic (otherwise replace $f$ by its affine conjugate) and denote
$$\mathcal{N}_d:=\{p\circ\exp: p\text{ is a monic polynomial of degree }d\}.$$
As for the exponential family \cite{SZ-Escaping}, for such functions the points escaping on rays can be described by their potential (or ``speed of escape'') and external address (or ``combinatorics of escape'', i.e.\ the sequence of dynamic rays containing the escaping orbit), for details see Theorem~\ref{thm:as_formula}. This is analogous to the B\"ottcher coordinates for polynomials where the points in the escaping set are encoded by their potential and external angle (another more general way to introduce ``B\"ottcher coordinates'' for trancendental entire functions of bounded type is described in \cite{LasseParaSpace}).

In complex dynamics it is very often important to study parameter spaces of holomorphic functions rather than each function separately. The most famous example of such parameter space is the Mandelbrot set. The simplest example in the transcendental world is the space of complex parameters $\kappa$ each associated to the function $e^z+\kappa$. However, for more general classes of transcendental entire functions, their parameter spaces have a less explicit form. 
\begin{defn}[Parameter space]
	Let $f$ be a transcendental entire function. Then the parameter space of $f$ is the set of transcendental entire functions $\varphi\circ f\circ\psi$ where $\varphi,\psi$ are \qc\ self-homeomorphisms of $\mathbb{C}$.
\end{defn}
For the exponential family $\mathcal{N}_1$, this definition coincides with the parametrization of $e^z+\kappa$ (up to affine conjugation), while for $d>1$, $\mathcal{N}_d$ is a disjoint union of different parameter spaces (again up to affine conjugation).

For every function $f_0\in\mathcal{N}_d$, our goal is to classify those functions in the parameter space of $f_0$ for which all its singular values escape on disjoint rays (that is, each dynamic ray contains at most one post-singular point; in particular, we avoid (pre-)periodic rays). This only slightly restricts the generality since it excludes countably many types of combinatorics out of uncountably many (of course, the case with (pre-)periodic rays is important on its own but is not the right initial point to tackle the general problem). The results in this direction for the exponential family $\mathcal{N}_1$ are proved \cite{FRS,MarkusParaRays,MarkusThesis} and reproved as a demonstration case in \cite{IDTT1}. Since the complement of the Mandelbrot set in $\mathbb{C}$ is exactly the set of parameters for which the critical value escapes on rays, the set we are classifying is roughly (contained in) a ``transcendental analogue to the complement of the Mandelbrot set'' for particular parameter spaces.

\begin{remark}
	After certain modifications of our toolbox it is possible to relax the condition of the disjointness of rays. Absence of this disjointness basically boils down to consideration of two possibilities (possibly appearing together): when two singular points are mapped after finitely many iterations onto the same dynamic ray, and when the ray containing a singular point is (pre-)periodic. The former case is rather a notational complication: the construction in Theorem~\ref{thm:invariant_subset} would need to take care of possible branches of rays. The latter case would require a slightly different description of homotopy types ``spiders'' (see Subsection~\ref{subsec:spiders}); in particular, one would need to introduce ``spider legs with infinitely many knees'' (at marked points), and adjust the definition of a ``leg homotopy word'' (Definition~\ref{defn:leg_homotopy_word}). Both cases would not require any essential additional analysis of the \tei\ space, but would rather overload Theorem~\ref{thm:invariant_subset} with a multi-level notation. Therefore, we consider the slightly ``simpler'' case of disjoint rays. 
	
	An alternative solution is discussed in \cite{IDTT4} where we use ``continuity'' of the parameter rays to extend the classification theorem to the case of escape on the (pre-)periodic rays.   
\end{remark}

For $d>1$, every $f_0\in\mathcal{N}_d$ has more than one singular orbit. If two singular points escape on disjoint orbits, it might happen that the (Euclidean) distance between them (as between sets) is equal to zero. This is a rather special case (see Theorem~\ref{thm:as_formula}) which requires an additional analysis (see the discussion after Theorem~\ref{thm:invariant_subset}). Therefore, we omit this case in this article (but fully solve it in \cite{IDTT3}) and prove the following theorem.

\begin{thm}[Classification Theorem]
	\label{thm:main_thm}
	Let $f_0\in\mathcal{N}_d$ be a transcendental entire function with singular values $\{v_i\}_{i=1}^m$. Let also $\{\underline{s}_i\}_{i=1}^m$ be exponentially bounded external addresses that are non-(pre-)periodic and non-overlapping, and $\{T_i\}_{i=1}^m$ be real numbers such that $T_i>t_{\underline{s}_i}$ and admitting only finitely many non-trivial clusters. Then in the parameter space of $f_0$ there exists a unique entire function $g=\varphi\circ f_0\circ\psi$ such that each of its singular values $\varphi(v_i)$ escapes on rays, has potential $T_i$ and external address $\underline{s}_i$.
	
	Conversely, every function in the parameter space of $f_0$ such that its singular values escape on disjoint rays and with finitely many non-trivial clusters is one of these. 
\end{thm}

Simply speaking, having some fixed parameter space, and a list of ``reasonable'' combinatorics and speeds of escape (given by external addresses and potentials), we can find in this parameter space a unique function whose singular values escape as prescribed.

The condition that external addresses are non-(pre-)periodic and non-overlapping basically means that the singular values of $g$ escape on disjoint rays (for details see Section~\ref{sec:escaping_set}). That $\{\underline{s}_i\}_{i=1}^m$ and $\{T_i\}_{i=1}^m$ admit only finitely many non-trivial clusters implies that the distance between orbits is strictly positive (see the discussion after Definition~\ref{defn:cluster}). The second paragraph of the Classification Theorem~\ref{thm:main_thm} is an immediate corollary of Theorem~\ref{thm:as_formula}, and we have added it to the statement in order to show that we indeed have a \emph{classification} of the functions in $\mathcal{N}_d$ with singular values escaping on rays (subject to some restrictions).

We are going to use the technical machinery developed in \cite{MyThesis,IDTT1}. It is based on a generalization of the approach in \emph{Thurston's Topological Characterization of Rational Functions} \cite{DH,HubbardBook2} for the case of \emph{infinitely} many marked points. In particular, we are going to use the same general strategy for the proof of Classification Theorem~\ref{thm:main_thm}. 
\begin{enumerate}
	\item Construct a (non-holomorphic) ``model map'' $f$ for which the singular values escape on rays with the desired potential and external address. It will define the Thurston $\sigma$-map acting on the (infinite-dimensional) \tei\ space $\mathcal{T}_f$ of the complement of post-singular set $P_f$ of $f$ (Definition~\ref{defn:tei_space}).
	\item Construct a compact subset $\mathcal{C}_f$ of $\mathcal{T}_f$ which is invariant under $\sigma$.
	\item Prove that $\sigma$ is strictly contracting  in the \tei\ metric on $\mathcal{C}_f$.
	\item Prove that $\sigma$ has the unique fixed point in $\mathcal{C}_f$, and this point corresponds to the entire function with the desired conditions on its singular orbits. 
\end{enumerate}

Items (1),(3) and (4) are proved analogously as in \cite{IDTT1}. Therefore, the ``heaviest'' item is the construction of $\mathcal{C}_f$. The main difference from the exponential case in \cite{IDTT1} is that an additional analysis of the parameter space is required (Section~\ref{subsec:prel_constructions}).

The invariant subset $\mathcal{C}_f$ is a substantial generalization of the corresponding construction in the exponential case \cite{IDTT1}: the biggest difficulty is to produce precise estimates so that $\mathcal{C}_f$ were invariant. 

However, on the superficial level $\mathcal{C}_f$ looks relatively simple and natural. For some initially chosen constant $\rho>0$, it is the set of points (in $\mathcal{T}_f$) represented by \qc\ maps $\varphi$ which can be obtained from identity via an isotopy $\varphi_u, u\in[0,1]$ such that: 
\begin{enumerate}
	\item for very point $z\in P_f$ inside of $\mathbb{D}_\rho(0)$, $\varphi_u(z)$ stay inside of $\mathbb{D}_\rho(0)$;
	\item for very point $z\in P_f$ outside of $\mathbb{D}_\rho(0)$, $\varphi_u(z)$ is contained in a small disk $U_z$ around $z$ so that the disks $U_z$ are mutually disjoint; 
	\item inside of $\mathbb{D}_\rho(0)$ the distances between marked points $\varphi_u(z)$ are bounded from below;
	\item the isotopy type of $\varphi_u$ relative points of $P_f$ inside of $\mathbb{D}_\rho(0)$ is not ``too complicated'', i.e.\ there are quantitative bounds on how many times the marked points ``twist'' around each other under the isotopy.
\end{enumerate}

Conditions (1) and (2) separate the complex plane into the two subsets: $\mathbb{D}_\rho(0)$ where we have not so much control on the behavior of $\varphi_u$ but have finitely many marked points, and the complement of $\mathbb{D}_\rho(0)$ where the homotopy information is trivial but we have infinitely many marked points.

Conditions (1)-(3) describe the position in the moduli space, while (4) encodes homotopy information of a point in the \tei\ space.

It is exactly the dynamic rays what allows us to store this homotopy information in an effective way and similar machinery can be generalized to more general transcendental entire functions admitting escape on rays. The explicit form of the functions that we consider helps to estimate better the behaviour of the marked points under the Thurston iteration. Possible generalizations on this side would require a better understanding of the local structure of the parameter spaces.

\subsubsection*{Structure of the article}

In Section~\ref{sec:escaping_set} we prove that for every function in $\mathcal{N}_d$, the points escaping on rays can be encoded via their potentials and external addresses (Theorem~\ref{thm:as_formula}).

In Section~\ref{sec:setup_and_contraction} we define the Thurston $\sigma$-map and show that it is strictly contracting on the set of asymptotically conformal points in the \tei\ space.

Afterwards, in Section~\ref{sec:id_type_and_spiders} we adapt the notions of $\id$-type maps and spiders and the associated techniques introduced in \cite{IDTT1} to our setting.

Finally, in Section~\ref{sec:inv_compact_subset} we construct the invariant compact subset (Theorems~\ref{thm:invariant_subset} and \ref{thm:compact_subset}), and prove the Classification Theorem~\ref{thm:main_thm}.

Appendix A contains a few properties of polynomials and functions from $\mathcal{N}_d$ that we occasionally use throughout the article.

\section{Escaping set of functions in $\mathcal{N}_d$}

\label{sec:escaping_set}

In this section we prove some preliminary technical results about the structure of the escaping set of the functions in $\mathcal{N}_d$. Within the section we assume that $d$ is fixed. We use the ideas from \cite{SZ-Escaping} where analogous results are formulated for the exponential family $\exp(z)+\kappa$ (i.e.\ for $\mathcal{N}_1$).

The first statement is that for functions in $\mathcal{N}_d$, as for the exponential family, the escaping points escape ``to the right''.

\begin{lmm}[Escape to the right]
	If $f\in\mathcal{N}_d$, then $f^n(z)\to\infty$ if and only if $\Re f^n(z)\to+\infty$.	
\end{lmm}
\begin{proof}
	$$\abs{f^n(z)}=\abs{p\circ\exp\circ f^{n-1}(z)}\to\infty\iff$$
	$$\abs{\exp\circ f^{n-1}(z)}\to\infty\iff $$
	$$\Re f^{n-1}(z)\to+\infty.$$	
\end{proof}

Now, we introduce certain domains that help to describe the combinatorics of the escaping points and play a similar role as the tracts in the study of transcendental entire functions of bounded type (in fact they are the tracts of $f^2$). By a slight abuse of terminology we also call them tracts.

\begin{lmm}[Definition and properties of tracts]
	\label{lmm:tracts}
	Let $f\in\mathcal{N}_d$, and let $r\in\mathbb{R}$ be such that the right half-plane $\mathbb{H}_r=\{z\in\mathbb{C}:\Re z>r\}$ does not contain any singular value of $f$. Then the preimage of $\mathbb{H}_r$ under $f$ has countably many path-connected components (called \emph{tracts}) having $\infty$ as a boundary point, so that for every tract $T$ the restriction of $f$ on $T$ is a conformal isomorphism.
	
	Moreover, for every $\epsilon\in(0,\pi/2d)$ we can choose $r,t^*,t_*\in\mathbb{R}$ such that
	\begin{enumerate}
		\item every tract $T$ is contained in a right-infinite rectangular strip
		$$[t^*,+\infty)\times\left[\frac{2\pi n}{d}-\frac{\pi}{2d}-\epsilon,\frac{2\pi n}{d}+\frac{\pi}{2d}+\epsilon\right],$$
		for some $n\in\mathbb{Z}$,
		\item every tract $T$ contains a right-infinite rectangular strip
		$$[t_*,+\infty)\times\left[\frac{2\pi n}{d}-\frac{\pi}{2d}+\epsilon,\frac{2\pi n}{d}+\frac{\pi}{2d}-\epsilon\right],$$
		for some $n\in\mathbb{Z}$.
	\end{enumerate}   
\end{lmm} 
\begin{proof}
	Let $T$ be a connected component of $f^{-1}(\mathbb{H}_r)$. Note that the restriction of $f$ on $T$ is a covering map. But since $\mathbb{H}_r$ is simply-connected, its cover $T$ must be simply-connected as well. This means that $f|_{\mathbb{H}_r}$ is a homeomorphism, and hence a conformal isomorphism.
	
	Proof of the second part of the lemma is an elementary calculus exercise.	
\end{proof}

Starting from this place we fix some choice of tracts, such that the conditions in the second part of Lemma~\ref{lmm:tracts} are satisfied. For every $n\in\mathbb{Z}$ denote by $T_n$ the tract containing the infinite ray $\{t_*+t+2\pi i n/d: t\geq 0\}$.

We are also going to make use of the following lemma.

\begin{lmm}[Exponential separation of orbits {\cite[Lemma~3.1]{RRRS}}]
	\label{lmm:separation}
	Let $f\in\mathcal{N}_d$ and $T$ be a tract of $f$ such that $f(T)=\mathbb{H}_r$. If $w,z\in T$ are such that $\abs{w-z}\geq 2$, then
	$$\abs{f(w)-f(z)}\geq \exp\left(\frac{\abs{w-z}}{8\pi}\right)\left(\min\{\Re f(w),\Re f(z)\}-r\right).$$
\end{lmm}

Note that in \cite{RRRS} the lemma is stated in a slightly different context (for logarithmic coordinates). Nevertheless, it also remains true in our setting.

In order to describe the escaping set of a function in $\nd$ we need the notions of a \emph{ray tail} and of a \emph{dynamic ray}.

\begin{defn}[Ray tails]
	\label{dfn:ray_tail}
	Let $f$ be a transcendental entire function. A \emph{ray tail} of $f$ is a continuous curve $\gamma:[0,\infty)\to I(f)$ such that for every $n\geq 0$ the restriction $f^n|_\gamma$ is injective with $\lim_{t\to\infty}f^n(\gamma(t))=\infty$, and furthermore $f^n(\gamma(t))\to\infty$ uniformly in $t$ as $n\to\infty$.
\end{defn}

\begin{defn}[Dynamic rays, escape on rays, endpoints]
	\label{dfn:dynamic_ray}
	A \emph{dynamic ray} of a transcendental entire function $f$ is a maximal injective curve $\gamma:(0,\infty)\to I(f)$ such that $\gamma|_{[t,\infty)}$ is a ray tail for every $t>0$.
	
	If a point $z\in I(f)$ belongs to a dynamic ray, we say that $z$ \emph{escapes on rays}(because in this case every iterate of the point belongs to a dynamic ray). 
	
	If there exists a limit $z=\lim_{t\to 0}\gamma(t)$, then we say that $z$ is an \emph{endpoint of the dynamic ray} $\gamma$.  
\end{defn}

From \cite{RRRS} we know that the escaping set of functions in $\nd$ is organized in form of the dynamic rays. Nevertheless, we reproduce the techniques from \cite{SZ-Escaping} for functions in $\mathcal{N}_d$ in order to obtain a more precise description of them.

Let $\gamma:[0,\infty)\to I(f)$ be a ray tail of $f$. Then for every $n\geq 0$ there is a number $t_n\geq 0$ and an integer $s_n$ such that $f^n\circ\gamma|_{[t_n,\infty)}$ is contained in the tract $T_{s_n}$. Moreover, all $t_n$ except finitely many are equal to $0$. This easily follows from the uniform escape of the ray tail and the fact that every connected component of the preimage under $f^2$ of a small (punctured) neighborhood of $\infty$ is contained in some tract. 

\begin{defn}[External address]
	\label{defn:external_address}
	Let $z\in I(f)$ be a point escaping either on rays or on endpoints of rays. We say that $z$ has \emph{external address} $\underline{s}=(s_0 s_1 s_2 ... )$ where $s_n\in\mathbb{Z}$, if either each $f^n(z)$ lies on a dynamic ray contained in $T_{s_n}$ near $\infty$ or each $f^n(z)$ is the endpoint of such ray.
	
	We also say that the dynamic ray either containing $z$ or having it as endpoint has the external address $\underline{s}$. 
\end{defn}

It is clear that the external address does not depend on a particular choice of $\mathbb{H}_r$ in the definition of tracts.

On the set of all external addresses we can consider the usual shift-operator $\sigma:(s_0 s_1 s_2 ... )\mapsto (s_1 s_2 s_3 ... )$. Two external addresses $\underline{s}_1$ and $\underline{s}_2$ are called \emph{overlapping} if there are integers $k,l\geq0$ so that $\sigma^k\underline{s}_1=\sigma^l\underline{s}_2$.

Now we define a function which allows to characterize the ``speed of escape'' of points in $I(f)$. We will refer to this function very often in different parts of the article.

\begin{defn}[$F(t)$]
	Denote by $F:\mathbb{R^+}\to\mathbb{R^+}$ the function $$F(t):=\exp(dt)-1.$$
\end{defn}

Note that $F=F_d$ implicitly depends on $d$ but we suppress this from notation. This will not cause any confusion because $d$ is the same within a parameter space.

\begin{remark}
	If $d=1$, then the function $F$ is the same as was used for the exponential family in \cite{SZ-Escaping}. In fact, for $d>1$ our function $F=F_d$ is conjugate to $F=F_1=\exp(t)-1$ through an $\mathbb{R}^+$-homeomorphism having asymptotics $\frac{t-\log d}{d}$ near $\infty$, i.e.\ they are replaceable with each other. Nevertheless, for computational reasons we prefer to use $F=F_d$.
\end{remark}

What one should know about $F$ is that its iterates grow very fast, as shown in the next elementary lemma.

\begin{lmm}[Super-exponential growth of iterates]
	For every $t,k>0$ we have $F^n(t)/e^{n^k}\to\infty$ as $n\to\infty$.
\end{lmm}

The next definition relates two last notions.

\begin{defn}[Exponentially bounded external address, $t_{\underline{s}}$]
	\label{dfn:exp_bdd_address}
	We say that the external address $\underline{s}=(s_0 s_1 s_2 ... )$ is \emph{exponentially bounded} if there exists $t>0$ such that $s_n/{F^n(t)}\to 0$ as $n\to\infty$. The infimum of such $t$ we denote by $t_{\underline{s}}$.	
\end{defn}

Note that if $t>t_{\underline{s}}$, then $s_n/{F^n(t)}^{1/k}\to 0$ where $k$ is any positive integer constant.

Next statement claims that no other type of external addresses can appear in $I(f)$.

\begin{lmm}
	\label{lmm:only_exp_bdd}
	Let $f\in\nd$. If $z\in I(f)$ escapes either on rays or endpoints and has the external address $\underline{s}=(s_0 s_1 s_2 ... )$, then $\underline{s}$ is exponentially bounded.	
\end{lmm}
\begin{proof}
	If $\Re z$ is big enough, then 
	$$\abs{f(z)}+1=\abs{(1+o(1))\exp(dz)}<$$
	$$\exp(\Re d(z+1/2))<F(\Re z+1)\leq F(\abs{z}+1).$$
	
	Hence $$\abs{\Im f^n(z)}+1\leq\abs{f^n(z)}+1<F^n(\abs{z}+1).$$
	
	So $\Im f^n(z)/F^n(t)\to 0$ for any $t>\abs{z}+1$ since $F^n(t_1)/F^n(t_2)\to 0$ as $n\to\infty$ for $t_2>t_1$.
\end{proof}

Now we prove the key theorem of this section which generalizes the results of \cite{SZ-Escaping} beyond $\mathcal{N}_1$.
\begin{thm}[Escape on rays in $\nd$]
	\label{thm:as_formula}
	Let $f\in\mathcal{N}_d$. Then for every exponentially bounded external address there exists a dynamic ray realizing it.
	
	If $\mathcal{R}_{\underline{s}}$ is a dynamic ray having an exponentially bounded external address $\underline{s}=(s_0 s_1 s_2 ... )$ and no strict forward iterate of $\mathcal{R}_{\underline{s}}$ contains a singular value of $f$, then $\mathcal{R}_{\underline{s}}$ is the unique dynamic ray having external address $\underline{s}$ and it can be parametrized by $t\in(t_{\underline{s}},\infty)$ so that 
	\begin{equation}
		\label{eqn:as_formula}
		\mathcal{R}_{\underline{s}}(t)=t+\frac{2\pi i s_0}{d} + O(e^{-t/2}),
	\end{equation}
	and
	$$f^n\circ \mathcal{R}_{\underline{s}}=\mathcal{R}_{\sigma^n \underline{s}}\circ F^n.$$
	
	The asymptotic bounds $O(\cdot)=O_{t,n}(\cdot)$ for $\mathcal{R}_{\sigma^n\underline{s}}(t)$ are uniform in $n\geq 0$ if we restrict the parametrization to a ray tail contained in $\mathcal{R}_{\underline{s}}$ and its images under $f^n$.
	
	If none of the singular values of $f$ escapes, then $I(f)$ is the disjoint union of dynamic rays and their escaping endpoints.
\end{thm}
\begin{proof}
	Fix some $r>0$ big enough so that $\abs{f'(z)}\geq 2$ on $f^{-1}(\mathbb{H}_r)$, \, $f^{-1}(\mathbb{H}_r)$ is a disjoint union of tracts $T_n, n\in\mathbb{Z}$, each $f|_{T_n}$ is a conformal homeomorphism, and let $L_n:\mathbb{H}_r\to T_n$ be its inverse.
	
	Further, let $\underline{s}=(s_0 s_1 s_2 ... )$ be an exponentially bounded external address. For $n\geq 1$ and $t\geq t_{\underline{s}}$ consider the functions	
	
	$$g_n(t):=L_{s_{0}}\circ L_{s_{1}}\circ ... \circ L_{s_{n-1}}\left(F^n (t)+\frac{2\pi i s_{n}}{d}\right),$$
	whenever they are defined.
	
	Fix some $t>t_{\underline{s}}$. For every integer $k\geq 0$ consider $$\delta_k(u):=L_{s_k}(F(u)+2\pi is_{k+1}/d)-(u+2\pi is_k/d)$$
	with $u\geq F^k(t)$. Since $\delta_k(u)$ is bounded, we have
	
	$$f\left(u+\frac{2\pi is_k}{d}+\delta_k(u)\right)=F(u)+\frac{2\pi is_{k+1}}{d}\iff$$
	$$(1+O(e^{-u})) e^{ d(u+2\pi is_k/d+\delta_k(u))}=(1+O(e^{-u/2}))e^{du}\iff$$
	$$\delta_k(u)=O(e^{-u/2}).$$
	
	The estimate above makes sense for all $u\geq F^k(t)$ whenever the index $k$ is bigger than some integer $N\geq 0$. Note that the estimate $O(\cdot)$ is uniform in $k$.
	
	Further, if $n\geq N$,
	
	$$f^N\circ g_{n+1}\circ F^{-N}(u)-f^N\circ g_n\circ F^{-N}(u)=$$
	$$L_{s_{N}}\circ ... \circ L_{s_{n}}\left(F^{n+1-N} (u)+\frac{2\pi i s_{n+1}}{d}\right)-L_{s_{N}}\circ ... \circ L_{s_{n-1}}\left(F^{n-N} (u)+\frac{2\pi i s_{n}}{d}\right)=$$
	$$L_{s_{N}}\circ ... \circ L_{s_{n-1}}\left(F^{n-N} (u)+\frac{2\pi i s_{n}}{d}+\delta_n\left(F^{n-N}(u)\right)\right)-L_{s_{N}}\circ ... \circ L_{s_{n-1}}\left(F^{n-N} (u)+\frac{2\pi i s_{n}}{d}\right).$$
	
	Hence 
	$$\left| f^N\circ g_{n+1}\circ F^{-N}(u)-f^N\circ g_n\circ F^{-N}(u)\right|<C \exp\left(-F^{n-N}(u)/2\right)$$
	since $L_{s_i}$ is contracting ($\abs{f'(z)}>2$ on tracts $T_n$).
	
	By Cauchy's criterion we see that for all $N$ big enough $f^N\circ g_n\circ F^{-N}$ converges to a continuous function $\mathcal{R}_{\sigma^N\underline{s}}(u)$ uniformly on $[F^N(t),\infty]$, and, moreover,
	$$\mathcal{R}_{\sigma^N\underline{s}}(u)=u+2\pi i s_N/d + O(e^{-u/2}).$$
	This is a ray tail having external address $\sigma^N\underline{s}$. After taking $N$ preimages of it under $f$ (and possibly restricting it to a bigger $t$ to avoid issues of the singular values), we obtain a ray tail having external address $\underline{s}$. It is contained in a dynamic ray, so every exponentially bounded external address is realized via some dynamic ray $\mathcal{R}_{\underline{s}}$.
	
	If none of the strict forward iterates of $\mathcal{R}_{\underline{s}}$ contains a singular value, then, by choosing $t$ arbitrarily close to $t_{\underline{s}}$, the function can be extended to $(t_{\underline{s}},\infty)$ (and even to $[t_{\underline{s}},\infty)$ if $s_n/{F^n(t_{\underline{s}})}\to 0$). 
	
	The formula~\ref{eqn:as_formula} for $\mathcal{R}_{\underline{s}}$ follows from the computations above (we can just consider big values of $u$). From the definition of $g_n$, we obtain the equality $f^n\circ \mathcal{R}_{\underline{s}}=\mathcal{R}_{\sigma^n \underline{s}}\circ F$. Injectivity of $\mathcal{R}_{\underline{s}}(t)$ follows from the asymptotic formula~\ref{eqn:as_formula}: points corresponding to different parameters $t$ have distance from each other bigger than one after iterating $f$ some finite number of times. 
	
	Uniqueness of the ray having an external address $\underline{s}$ follows from strict contraction of $f^{-1}$ on $\mathbb{H}_r$: any two such rays would have had a common part near $\infty$. Since they do not contain a critical value, they cannot ``branch'', so must coincide. 
	
	Now our goal is to show that the parametrization on $(t_{\underline{s}},\infty)$ is the parametrization of the whole dynamic ray.  
	
	It is enough to prove that whenever two points are on the same ray $\mathcal{R}_{\underline{s}}$, then the real part of one of them grows faster than $F^n(t)$ for some $t>t_{\underline{s}}$. Indeed, because of the strict contraction of $f^{-1}$ on $\mathbb{H}_r$, by the standard argument it would follow that this point belongs to the parametrized part of $\mathcal{R}_{\underline{s}}$.
	
	Recall that in Lemma~\ref{lmm:only_exp_bdd} we proved that if $z\in I(f)$ and all iterates lie outside of some bounded set, then $\abs{\Im f^n(z)}+1<F^n(\abs{z}+1)$. But this means that if $z\in\mathcal{R}_{\underline{s}}$, then $\abs{z}+1\geq t_{\underline{s}}$. Analogously, for every $n\geq 0$,
	$$\abs{f^n(z)}+1\geq t_{\sigma^n\underline{s}}=F^n(t_{\underline{s}}).$$
	Taking logarithm of both sides of the inequality for $n=k+1$ and assuming $t_{\underline{s}}>0$, we obtain that for every $k\geq 0$
	$$\Re f^k(z)\geq F^k(t_{\underline{s}})+o(1)$$
	where $o(1)\to 0$ as $k\to\infty$.	
	
	Now, pick $z,w\in \mathcal{R}_{\underline{s}}$. Without loss of generality we might assume that they escape inside of the same sequence of tracts (otherwise switch to their iterates). The distance between them along the orbits has to be unbounded (otherwise they coincide by the usual contraction argument). But then without loss of generality we may assume that $
	\abs{f^n(w)-f^n(z)}>2$ for $n\geq 0$.
	
	From Lemma~\ref{lmm:separation} we have 	
	$$\abs{f^{n+1}(w)-f^{n+1}(z)}\geq \exp\left(\frac{\abs{f^n(w)-f^n(z)}}{8\pi}\right)\left(\min\{\Re f^{n+1}(w),\Re f^{n+1}(z)\}-r\right).$$
	
	Consider first the case  $t_{\underline{s}}>0$. From the two inequalities above,
	$$\abs{f^{n+1}(w)-f^{n+1}(z)}\geq \exp\left(\frac{\abs{f^n(w)-f^n(z)}}{8\pi}\right)\left(F^{n+1}(t_{\underline{s}}) +o(1)-r\right).$$
	It follows, that the difference $\varDelta_n=\abs{f^n(w)-f^n(z)}$ grows faster than $F^n(t_{\underline{s}}+\epsilon)$ as $n\to\infty$ for some $\epsilon>0$. Indeed, first note that the ratio $\varDelta_n/F^n(t_{\underline{s}})$ tends to $\infty$, then use a weaker inequality $\varDelta_{n+1}\geq\exp(\varDelta_n/8\pi)$ to estimate the order of growth.
	
	This implies that the sequence $\{\max(\abs{f^n(z)},\abs{f^n(w)})\}_{n=1}^{\infty}$ escapes to $\infty$ faster than $\{F^n(t_{\underline{s}}+\epsilon)\}$ for some $\epsilon>0$. From the definition of $t_{\underline{s}}$, the same holds for the sequence $\{\max(\Re f^n(z),\Re f^n(w))\}_{n=1}^{\infty}$. Thus, for big enough $n$ at least one point among $f^n(z)$ and $f^n(w)$ has in its $2\pi/d$ neighbourhood a point of the parametrized part of the ray $\mathcal{R}_{\sigma^n\underline{s}}$ having potential bigger than $F^n(t_{\underline{s}}+\epsilon)$. By the usual contraction and pull-back argument this is possible only if at least one of the points $z,w$ lies on the parametrized part of $\mathcal{R}_{\underline{s}}$.
	
	Now, let $t_{\underline{s}}=0$. Then for all $n$ big enough,	
	$$\left|f^{n+1}(w)-f^{n+1}(z)\right|\geq \exp\left(\frac{\abs{f^n(w)-f^n(z)}}{8\pi}\right),$$
	which implies that the sequence $\{\max(\abs{f^n(z)},\abs{f^n(w)})\}_{n=1}^{\infty}$ escapes to $\infty$ faster than $\{F^n(\epsilon)\}$ for some $\epsilon>0$. The rest of the proof is identical to the case $t_{\underline{s}}>0$.
	
	Finally, we have to show that if none of the singular values of $f$ escapes, then $I(f)$ is the union of dynamic rays and their escaping endpoints. This is a well known fact even for more general families of functions \cite{RRRS}. We can argue as follows. From \cite[Theorem~4.7]{RRRS} we know that every escaping point lands after finitely many iterations either on a dynamic ray or on an endpoint. Since rays do not contain singular values, they can be pulled back by $f$ at their full length (together with endpoints when exist), and their preimages are dynamic rays as well. This finishes the proof of the theorem.
\end{proof}

Theorem~\ref{thm:as_formula} allows us to define the notion of potential.

\begin{defn}[Potential]
	Let $f\in \mathcal{N}_d$ and $z$ escapes on rays with external address $\underline{s}$. We say that $t$ is the \emph{potential} of $z$ if for $n\to\infty$ 
	$$\abs{f^n(z)-F^n(t)-2\pi is_n/d}\to 0.$$	
\end{defn}

From Theorem~\ref{eqn:as_formula} follows that every point escaping either on rays has a potential, and different points on the same ray have different potentials.

Next lemma will help us to deal with some technicalities later on. It says that all far enough iterates of ray tails have a strictly increasing real part. We will need this to prove that these ray tails have the ``trivial'' homotopy type (Proposition~\ref{prp:constant_homotopy_near_infinity}), even after some perturbations of marked points. Note that such a statement was not needed in \cite{IDTT1} for the exponential family because the ``triviality'' of the homotopy types was evident.

\begin{lmm}[Ray tails are monotone near $\infty$]
	\label{lmm:monotonicity_of_rays}
	Let $\mathcal{R}_{\underline{s}}=\mathcal{R}_{\underline{s}}(t)$ be a dynamic ray of $f$ parametrized by potential, and let $t'$ be the potential of a point on the ray $\mathcal{R}_{\underline{s}}$. Then there exists $N>0$ such that for all $n>N$ the real part of $f^n(\mathcal{R}_{\underline{s}}|_{[t',\infty]})$ is strictly increasing.
\end{lmm}
\begin{proof}
	Recall that according to our earlier agreement $f$ maps each tract $T_n$ conformally onto $\mathbb{H}_r$ for some $r\in\mathbb{R}$.
	
	Let $R=\mathcal{R}_{\underline{s}}([t',+\infty])$ and $s=(s_0 s_1 s_2...)$. For every integer $n\geq 0$ define 
	$$Q_n:=\left[F^n(t')-\frac{1}{2},+\infty\right)\times\left[\frac{2\pi i s_n}{d} -\frac{1}{2^n}, \frac{2\pi i s_n}{d}+\frac{1}{2^n}\right].$$
	
	First, we prove that for all $n$ big enough holds:
	\begin{enumerate}
		\item $f^n(R)\subset Q_n\subset T_{s_n}$,
		\item if $z\in Q_n$, then 
		$$\abs{\arg f'(z)}<\frac{1}{2^n},$$
		where we consider the branch of $\arg$ with range in $(-\pi,\pi]$, 
		\item if $Q_n^{-1}$ is the connected component of $f^{-1}(Q_{n+1})$ containing $f^n(t')$, then $$Q_n^{-1}\subset Q_n.$$
	\end{enumerate} 
	
	\textbf{(1)} The first and the second inclusions follow immediately from the asymptotic formula~\ref{eqn:as_formula} and  Lemma~\ref{lmm:tracts}, respectively.
	
	\textbf{(2)} Note that if $z=x+iy\in Q_n$, then
	$$f'(z)=(p\circ\exp)'(z)=p'(e^z)e^z=$$
	$$d e^{dz}\left(1+O(e^{-z})\right)=d e^{dx} e^{idy}\left(1+O(e^{-x})\right)=$$
	$$d e^{dx} e^{idy}\left(1+O(e^{-F^n(t')})\right),$$
	where $O(\cdot)$ is uniform in $n$.
	
	Then for big $n$ due to the asymptotic formula~\ref{eqn:as_formula}
	$$\abs{\arg f'(z)}=\abs{\arg\left(1+O(e^{-F^n(t')/2})\right)+\arg \left(1+O(e^{-F^n(t')})\right)}<\frac{1}{2^n}.$$
	\textbf{(3)} Consider only $n$ big enough so that:
	\begin{itemize}
		\item $Q_n\subset\mathbb{H}_r$,
		\item for every $z\in f^n(R)$ we have 
		$$\mathbb{D}_{\frac{4}{\abs{f'(z)}}}(z)\subset Q_n.$$		
	\end{itemize}
	
	The second condition is possible due to the asymptotic formula~\ref{eqn:as_formula} and the fact that for some constant $C>0$ holds
	$$\frac{4}{\abs{f'(z)}}<\frac{C}{\exp(F^n(t'))}.$$
	
	If $n$ is big enough, then from the asymptotic formula~\ref{eqn:as_formula} follows that for every $w\in Q_{n+1}$ there exists a point $a\in f^n(R)\subset Q_n$ such that $$\abs{w-f(a)}<1.$$
	Next, since $f$ is univalent on every tract, from Koebe $1/4$ theorem follows that
	$$f(\mathbb{D}_{\frac{4}{\abs{f'(a)}}}(a))\supset\mathbb{D}_1(f(a))\ni w.$$
	Hence $Q_n$ contains a point $z$ such that $f(z)=w$.
	
	We have proven the conditions \textbf{(1)-(3)} and are ready to finish the proof of the lemma.
	
	Let $N>1$ be such that for $n\geq N$ conditions \textbf{(1)-(3)} are satisfied. Take $t_1,t_2$ such that $t_2>t_1\geq t'$, and for $n>N$ let
	$\gamma_n:[0,1]\to\mathbb{C}$ be the straight line segment joining $\mathcal{R}_{\sigma^n\underline{s}}(t_1)$ to $\mathcal{R}_{\sigma^n\underline{s}}(t_2)$ so that
	$$\gamma_n(u)=\mathcal{R}_{\sigma^n\underline{s}}(t_1)+u(\mathcal{R}_{\sigma^n\underline{s}}(t_2)-\mathcal{R}_{\sigma^n\underline{s}}(t_1)).$$
	The segment $\gamma_n$ is evidently contained in $Q_n$. 
	
	It follows from the asymptotic formula~\ref{eqn:as_formula} that we can find $m\geq N$ such that
	$$\abs{\arg \gamma_m}<\frac{1}{2},$$
	in particular, $\Re \mathcal{R}_{\sigma^m\underline{s}}(t_2)>\Re \mathcal{R}_{\sigma^m\underline{s}}(t_1)$.
	
	Consider the smooth curve $\tilde{\gamma}(u):=f^{N-m}\circ\gamma_m(u)$ joining $\mathcal{R}_{\sigma^N\underline{s}}(t_2)$ and $\mathcal{R}_{\sigma^N\underline{s}}(t_1)$. Because of \textbf{(2)}, for every $u\in[0,1]$ we have
	$$\arg\tilde{\gamma}'(u)<\frac{1}{2^{N}}+\frac{1}{2^{N+1}}+...+\frac{1}{2^{m-1}}+\abs{\arg\gamma_m}<\frac{1}{2}+\frac{1}{2}=1<\frac{\pi}{2}.$$
	Hence $\Re\tilde{\gamma}$ is strictly increasing, and $\Re \mathcal{R}_{\sigma^N\underline{s}}(t_2)>\Re \mathcal{R}_{\sigma^N\underline{s}}(t_1)$.
	The same computations work for every $n>N$ instead of $N$, i.e.\ $\Re \mathcal{R}_{\sigma^n\underline{s}}(t_2)>\Re \mathcal{R}_{\sigma^n\underline{s}}(t_1)$.	
\end{proof}

\section{Setup of the Thurston iteration and contraction}

\label{sec:setup_and_contraction}

In this section we define Thurston's $\sigma$-map and discuss its properties, in particular, that it is (uniformly) strictly contracting on the set of the asymptotically conformal points in the \tei\ space. The definitions are analogous to those in \cite{IDTT1} for the exponential family, so we provide only a brief exposition. See \cite{IDTT1} for a more detailed discussion.

We show how to construct a quasiregular function $f$ modeling some escaping behavior of singular values. Let $f_0\in\mathcal{N}_d$, and $v_1,...,v_m$ be the finite singular values of $f_0$. Further, let $\mathcal{O}_1=\{a_{1j}\}_{j=0}^\infty,...,\mathcal{O}_m=\{a_{mj}\}_{j=0}^\infty$ be some orbits of $f_0$ escaping on disjoint rays $\mathcal{R}_{ij}$. Denote by $R_{ij}$ the part of the ray $\mathcal{R}_{ij}$ from $a_{ij}$ to $\infty$ (including $a_{ij}$ itself) and assume that the singular values $v_1,...,v_m$ do not belong to any of the $R_{ij}$'s.

Now, we describe the construction of a \emph{capture map}. It is a carefully chosen \qc\ homeomorphism of $\mathbb{C}$ that is equal to the identity outside of a bounded set and maps singular values of $f_0$ to the points $\{a_{i0}\}$. 

More precisely, choose some bounded Jordan domain $U\subset\mathbb{C}\setminus\bigcup_{\substack{i=\overline{1,m}\\j=\overline{1,\infty}}} R_{ij}$ containing all singular values of $f_0$ and the first point $a_{i0}$ on each orbit $\mathcal{O}_i$. Define an isotopy $c_u:\mathbb{C}\to\mathbb{C}, u\in[0,1]$ through \qc\ maps, such that $c_0=\id$, $c_u=\id$ on $\mathbb{C}\setminus U$, and for each $i=\overline{1,m}$ we have $c_1(v_i)=a_{i0}$. Denote $c=c_1$. Thus, the capture map $c$ is a \qc\ homeomorphism mapping singular values to the first points on the orbits $\mathcal{O}_i$ and ``not spoiling'' the dynamics on the chosen orbits $\mathcal{O}_i$.

\begin{remark}
	Note that the choice of the capture is not unique, so we just pick one of them. It will be shown in the proof of Classification Theorem~\ref{thm:main_thm} that our results do not depend on a particular choice of the capture.
\end{remark}

Define a function $f:=c\circ f_0$. It is a quasiregular function whose singular orbits coincide with $\{\mathcal{O}_i\}_{i=1}^m$. We use the standard notation for the post-singular set of $f$: $$P_f:=\cup_{i=1}^m\mathcal{O}_i.$$ It is also common to call $P_f$ the set of \emph{marked points}.

Thurston's $\sigma$-map is supposed to act on $\mathcal{T}_f$, the \tei\ space of $\mathbb{C}\setminus P_f$.

\begin{defn}[\tei\ space of $\mathbb{C}\setminus P_f$]
	\label{defn:tei_space}
	The \emph{\tei\ space} $\mathcal{T}_f$ of the Riemann surface $\mathbb{C}\setminus P_f$ is the set of quasiconformal homeomorphisms of $\mathbb{C}\setminus P_f$ modulo post-composition with an affine map and isotopy relative $P_f$.
\end{defn}

The map
$$\sigma:[\varphi]\in\mathcal{T}_f\mapsto[\tilde{\varphi}]\in\mathcal{T}_f$$
is defined in the standard way via Thurston's diagram.

\begin{center}
	\begin{tikzcd}
		\mathbb{C},P_f \arrow[r, "{\tilde{\varphi}}"] \arrow[d, "f=c\circ f_0"]	& \mathbb{C},\tilde{\varphi}(P_f) \arrow[d, "g"] \\
		\mathbb{C},P_f \arrow[r, "{\varphi}"] & \mathbb{C},\varphi(P_f)
	\end{tikzcd}
\end{center}
\vspace{0.5cm}

More precisely, let $[\varphi]$ be a point in $\mathcal{T}_f$ where $\varphi:\mathbb{C}\to\mathbb{C}$ is quasiconformal. Then due to the Measurable Riemann Mapping Theorem there exist a unique (up to the postcomposition with an affine map) \qc\ map $\tilde{\varphi}:\mathbb{C}\to\mathbb{C}$ such that $g=\varphi\circ f\circ\tilde{\varphi}^{-1}$ is an entire function. We define $\sigma[\varphi]:=[\tilde{\varphi}]$. So $\sigma$ is a continuous map acting on $\mathcal{T}_f$. For more details and the proof that this setup is well defined we refer the reader to \cite{IDTT1}. For the classical setup we suggest \cite{DH,HubbardBook2}. 

As was mentioned in the Introduction, we are interested in fixed points of $\sigma$ with the expectations that the entire map $g$ on the right hand side of Thurston's diagram would be the entire function that we need in the Classification Theorem~\ref{thm:main_thm}.

As for the exponential case in \cite{IDTT1}, we will need the fact that the $\sigma$-map is strictly contracting on the subspace of asymptotically conformal points of $\mathcal{T}_f$.

\begin{defn}[Asymptotically conformal points \cite{Gardiner}]
	\label{defn:as_conformal}
	A point $[\varphi]\in\mathcal{T}_f$ is called \emph{asymptotically conformal} if for every $\epsilon>0$ there is a compact set $C\subset\mathbb{C}\setminus P_f$ and a representative $\psi\in[\varphi]$ such that $\abs{\mu_\psi}<\epsilon$ a.e.\ on $(\mathbb{C}\setminus P_f)\setminus C$.
\end{defn}

Next theorem is an immediate corollary from \cite[Lemma~4.1]{IDTT1} and \cite[Lemma~4.3]{IDTT1}.

\begin{thm}[$\sigma$ is strictly contracting on as.\ conf.\ subset]
	\label{thm:sigma_strictly_contracting}
	Let $f=c\circ f_0$ be the quasiregular function defined earlier in this section. Then the associated $\sigma$-map is invariant and strictly contracting on the subset of the asymptotically conformal points in $\mathcal{T}_f$.
\end{thm}

\section{$\Id$-type maps and spiders}
\label{sec:id_type_and_spiders}

In this section we mainly adjust the machinery developed in \cite{IDTT1} for the needs of our current constructions. 

\subsection{$\Id$-type maps}

Let $f=c\circ f_0$ be the quasiregular function defined in Section~\ref{sec:setup_and_contraction} where $f_0=p\circ\exp$. We start with three definitions.

\begin{defn}[Standard spider]
	$S_0=\cup_{i,j} R_{ij}$ is called the \emph{standard spider} of $f$.
\end{defn}

\begin{defn}[$\Id$-type maps]
	\label{defn:id_type}
	A quasiconformal map $\varphi:\mathbb{C}\to \mathbb{C}$ is \idt\ if there is an isotopy $\varphi_u:\mathbb{C}\to\mathbb{C},\ u\in [0,1]$ such that ${\varphi_1=\varphi},\ \varphi_0=\id$ and $\abs{\varphi_u(z)-z}\to 0$ uniformly in $u$ as $S_0\ni z\to \infty$.
\end{defn}

\begin{defn}[$\Id$-type points in $\mathcal{T}_f$]
	We say that $[\varphi]\in\mathcal{T}_f$ is \idt\ if $[\varphi]$ contains an $\id$-type map.	
\end{defn}

As for the exponential case \cite{IDTT1}, one can replace $S_0$ in Definition~\ref{defn:id_type} by solely the endpoints of $S_0$, i.e.\ $P_f$. It is easy to show that this would define the same subset \idt\ points in $\mathcal{T}_f$. However, Definition~\ref{defn:id_type} is slightly more convenient for our needs.

\begin{defn}[Isotopy \idt\ maps]
	We say that $\psi_u$ is an \emph{isotopy \idt\ maps} if $\psi_u$ is an isotopy through maps \idt\ such that $\abs{\psi_u(z)-z}\to 0$ uniformly in $u$ as $S_0\ni z\to \infty$.
\end{defn}

The identity and $c^{-1}$ are clearly \idt, as well as a composition $\varphi\circ c$ of the capture $c$ with any $\id$-type map.

The following theorem claims that the $\sigma$-map is invariant on the subset of $\id$-type points in $\mathcal{T}_f$ in analogy with \cite[Theorem~5.5]{IDTT1}.

\begin{thm}[Invariance of $\id$-type points]
	\label{thm:pullback_of_id_type}
	If $[\varphi]$ is \idt, then $\sigma[\varphi]$ is \idt\ as well.
	
	More precisely, if $\varphi$ is \idt, then there is a unique $\id$-type map $\hat{\varphi}$ such that $\varphi\circ f\circ\hat{\varphi}^{-1}$ is entire.
	
	Moreover, if $\varphi_u$ is an isotopy of $\id$-type maps, then the functions $\varphi_u\circ f\circ\hat{\varphi}_u^{-1}$ have the form $p_u\circ\exp$ where $p_u$ is a monic polynomial with coefficients depending continuously on $u$.
\end{thm}
\begin{remark}
	Conceptually the proof is the same as of \cite[Theorem~5.5]{IDTT1} but since we have a more general family of functions, there is an upgrade on the level of formulas.
\end{remark}
\begin{proof}
	First, since $\hat{\varphi}$ is defined up to a postcomposition with an affine map, and $S_0$ contains continuous curves joining finite points to $\infty$, uniqueness is obvious. We only have to prove existence.
	
	Since $c^{-1}$ is also \idt, $\varphi$ can be joined to $c^{-1}$ through an isotopy \idt\ maps: we can simply take the concatenation of two isotopies to identity ($\varphi\sim\id$ and $\id\sim c^{-1}$). Let $\psi_u, u\in [0,1]$ be this isotopy with $\psi_0=c^{-1}$ and $\psi_1=\varphi$, and let $\tilde{\psi}_u$ be the unique isotopy of \qc\ self-homeomorphisms of $\mathbb{C}$, so that every $\tilde{\psi}_u$ is the solution of the Beltrami equation
	$$\frac{\partial_{\overline{z}}\tilde{\psi}_u}{\partial_z\tilde{\psi}_u}=\frac{\partial_{\overline{z}}(\psi\circ f)}{\partial_z (\psi\circ f)},$$
	fixing $0$ and $1$ (provided by the Measurable Riemann Mapping Theorem). Note that $\tilde{\psi}_0=\id$ because $\psi_0\circ f=c^{-1}\circ c\circ f_0=f_0$.
	Then the maps $h_u=\psi_u\circ f \circ \tilde{\psi}_u^{-1}$ will have the form 	
	$$h_u(z)=\beta_u q_u\circ\exp(\alpha_u z),$$ 	
	where $\alpha_u,\beta_u\in\mathbb{C}\setminus\{0\}$ and $q_u$ is a monic polynomial of degree $d$: a conformal branched covering of $\mathbb{C}$ of degree $d$ is necessarily a polynomial of the same degree, and a conformal branched covering of the punctured plane is the exponential (up to compositions with affine maps). Moreover, $\alpha_u,\beta_u$ and the coefficients of $q_u$ depend continuously on $u$.
	
	Now, consider the isotopy $\hat{\psi}_u(z):=\alpha_u\tilde{\psi}_u(z) +\log(\beta_u)/d$, where the branch of the logarithm is chosen so that $\psi_0\circ f \circ \hat{\psi}_0^{-1}=f$. Then
	$$g_u:=\psi_u\circ f \circ \hat{\psi}_u^{-1}=p_u\circ\exp$$ 
	is a homotopy of entire functions so that $g_u\in\mathcal{N}_d$, $p_0=p$ and the polynomial coefficients are continuous in $u$.	
	
	Since the coefficients of $p_u$ are continuous, they are bounded on $[0,1]$. Thus $p_u(z)=z^d(1+O(1/z))$ where $O(.)$ is a bound that is uniform in $u$.
	
	Now, let $z\in S_0$, and assume that $\abs{z}$ is big enough (so that all subsequent computations are correct). Since all $\psi_u$ are \idt, $\abs{\psi_u\circ f(z)-f(z)}<1$ for every $u$. Let $w_u=w_u(z):=\exp (\hat{\psi}_u (z))$. Then 
	
	$$\abs{\psi_u\circ f(z)-f(z)}=\abs{\psi_u\circ f(z)-\psi_0\circ f(z)}=\abs{p_u(w_u)-p(w_0)}<1.$$
	Hence	
	$$w_u^d(1+O(1/w_u))=w_0^d(1+O(1/w_0)).$$	
	Due to the continuity of $\hat{\psi}_u$ and since $\hat{\psi}_0=\id$, we have	
	$$w_u=w_0 (1+o(1)),$$
	where $o(\cdot)\to 0$ uniformly in $u$ as $z\to\infty$. Hence
	$$\exp(\hat{\psi}_u (z))=\exp(\hat{\psi}_0 (z))(1+o(1)).$$
	Again, due to the continuity of $\hat{\psi}_u$	
	$$\abs{\hat{\psi}_u (z) - \hat{\psi}_0 (z)}=\abs{\hat{\psi}_u (z)-z}=o(1).$$	
	The first statement of the theorem is then proven with $\hat{\varphi}:=\hat{\psi}_1$.
	
	The last paragraph of the theorem follows immediately from the computations.
\end{proof}

Theorem~\ref{thm:pullback_of_id_type} implies that the notion of external address is preserved when we iterate $\id$-type maps, in a sense that under proper normalization, i.e.\ when we choose the map $\hat{\varphi}$, the images of dynamic rays under $\hat{\varphi}$ preserve original asymptotic straight lines. 

\begin{remark}
	In the sequel we keep using the hat-notation from Theorem~\ref{thm:pullback_of_id_type}. That is, $\hat{\varphi}$ denotes the unique $\id$-type map so that $\varphi\circ f\circ\hat{\varphi}^{-1}$ is entire. Due to Theorem~\ref{thm:pullback_of_id_type} this notation makes sense whenever $\varphi$ is \idt.	
\end{remark}

\subsection{Spiders}

\label{subsec:spiders}

In this subsection we discuss infinite-legged spiders and their properties. Everything is the ad hoc adaptation of the machinery developed in \cite{IDTT1}.

\begin{defn}[Spider]
	\label{defn:spider}
	An image $S_{\varphi}$ of the standard spider $S_0$ under an $\id$-type map $\varphi$ is called a \emph{spider}.	
\end{defn}

\begin{defn}[Spider legs]
	The image of a ray tail $R_{ij}$ under a spider map is called a \emph{leg} (of a spider).
\end{defn}

Let $V=\{w_n\}\subset\mathbb{C}$ be a finite set. Further, let $\gamma:[0,\infty]\to\hat{\mathbb{C}}$ be a curve such that $\gamma(0)\in V$, $\gamma(\infty)=\infty$, $\gamma|_{\mathbb{R}^+}\subset\mathbb{C}\setminus V$ and $\Re\gamma(t)\to+\infty$ as $t\to\infty$. On the set of all such pairs $(V,\gamma)$ one can consider a map
$$W: (V,\gamma)\mapsto W(V,\gamma)\in F(V),$$
where $F(V)$ is the free group on $V$ (every element of $F(V)$ is a finite word with symbols from the alphabet $\{w_1^{\pm 1},\dots,w_n^{\pm1}\}$). This map uniquely encodes the homotopy type of $\gamma$ (with fixed endpoints) in $\mathbb{H}_r\setminus V$ where $\mathbb{H}_r=\{z\in\mathbb{C}:\Re z>r\}$ contains $V$ and $\gamma$ (a particular value of $r$ is irrelevant). Roughly speaking, we homotope $\gamma$ into a concatenation of a horizontal straight ray from $\gamma(0)$ to $+\infty$, and a loop with the base point at $\infty$, and by $W(V,\gamma)$ denote the representation of the loop via ``straight horizontal'' generators of the fundamental group of $(\mathbb{H}_r\setminus V)\cup\infty$. For the details of the construction we refer the reader to \cite[Subsection~6.1]{IDTT1}. We are going to use this type of information for every leg of a spider, in order to obtain a tame description of the $\id$-type points in $\mathcal{T}_f$.

For every pair of $i\in \{1,2,...,m\},j\geq 0$ denote
$$\mathcal{O}_{ij}:=\{a_{kl}\in P_f: l<j\text{ or } l=j,k\leq i\}.$$

\begin{defn}[Leg homotopy word]
	\label{defn:leg_homotopy_word}
	Let $S_\varphi$ be a spider. Then the \emph{leg homotopy word} of a leg $\varphi(R_{ij})$ is 
	$$W_{ij}^\varphi:=W(\varphi(\mathcal{O}_{ij}),\varphi(R_{ij})).$$
\end{defn}

Next theorem provides an estimate of how the leg homotopy words changes under the Thurston iteration.

\begin{thm}[Combinatorics of a preimage]
	\label{thm:homotopy_type_under_pullback}
	Let $\varphi$ be \idt. Then $$\abs{W_{ij}^{\hat{\varphi}}}<A(j+1)^4 \max\{1,\abs{W_{i(j+1)}^\varphi}\},$$ where $A$ is a positive real constant.
\end{thm}
\begin{proof}
	This is a restatement of \cite[Theorem~6.16]{IDTT1} for our context. That is, when (in the notation of \cite[Theorem~6.16]{IDTT1}) $V=\varphi(\mathcal{O}_{i(j+1)})$, $\gamma=\varphi(R_{i(j+1)})$ and $W(V,\gamma)=W_{i(j+1)}^\varphi$. Then if $\tilde{\gamma}=\hat{\varphi}(R_{ij})$,
	$$\abs{W_{ij}^{\hat{\varphi}}}<42d(m(j+1)+i)^4 (\abs{W_{i(j+1)}^\varphi}+1)<A(j+1)^4\max\{\abs{W_{i(j+1)}^\varphi},1\}.$$
\end{proof}

Finally, we introduce a special equivalence relation of spiders, which coincides with the \tei\ equivalence of the associated $\id$-type maps.

\begin{defn}[Projective equivalence of spiders]
	\label{defn:proj_equiv}
	We say that two spiders $S_\varphi$ and $S_\psi$ are projectively equivalent if for all pairs $i,j$ we have $\varphi(a_{ij})=\psi(a_{ij})$ and $W_{ij}^\varphi=W_{ij}^\psi$.
\end{defn}

\begin{thm}[Projective equivalence of spiders is \tei\ equivalence]
	\label{thm:W_define_teich_point}
	Two spiders $S_\varphi$ and $S_\psi$ are projectively equivalent if and only if $[\varphi]=[\psi]$, i.e.\ $\varphi$ is isotopic to $\psi$ relative ${P_f}$.
\end{thm}
\begin{proof}
	If we replace the double index $i,j$ by $mj+i-1$, then we are in the setting of \cite[Theorem~6.28]{IDTT1} (see also \cite[Remark~6.29]{IDTT1}).
\end{proof}

\section{Invariant compact subset}

\label{sec:inv_compact_subset}
\subsection{Preliminary constructions}
\label{subsec:prel_constructions}

As earlier we consider the \qc\ function $f=c\circ f_0$ with $m$ singular orbits $\mathcal{O}_i=\{a_{ij}\}$ constructed in Section~\ref{sec:setup_and_contraction} and the associated $\sigma$-map. Denote also by $t_{ij}$ the potential of $a_{ij}$ and by $s_{ij}$ the index of the tract containing $R_{ij}$ near $\infty$. The goal of this subsection is to prove a few preliminary lemmas that provide estimates on the pullback-spider $S_{\hat{\varphi}}$ based on different types of information about a given spider $S_\varphi$. 

Let $\mathcal{P}=\{t_i\}_{i=1}^\infty$ be the set of all potentials of points in $P_f$ ordered so that $t_i<t_{i+1}$. Obviously, to each potential correspond at most $m$ points of $P_f$, and $\mathcal{P}$ is ``eventually periodic'' in a sense that there exists a positive integer $T\leq m$ such that for all $t_i$ big enough $t_{i+T}=F(t_i)$. Note that $T$ can be strictly less than $m$ if points from different orbits have equal potentials.

Also we define a set $$\mathcal{P}':=\{\rho_i:\rho_i=\frac{t_i+t_{i+1}}{2}\}.$$ It will come into action a bit later.

Next definition describes a structure that might appear in the case when we have more than one singular orbit. 

\begin{defn}[Cluster]
	\label{defn:cluster}
	We say that $a_{ij}$ and $a_{kl}$ are in the same \emph{cluster $Cl(t,s)$} if they have the same potential $t=t_{ij}=t_{kl}$ and belong to the ray tails $R_{ij}$ and $R_{kl}$ contained in the same tract $T_s=T_{ij}=T_{kl}$ near $\infty$.
	
	A cluster is called non-trivial if it contains more than one point of $P_f$.
\end{defn}

The set $P_f$ is a disjoint union of clusters and each cluster contains at most $m$ points. As can be seen from the asymptotic formula~\ref{eqn:as_formula}, for points with big potentials, being in the same cluster implies that the distance between them is very small, whereas the distance between any pair of clusters is bounded from below.

Denote by $\SV(f)$ the singular values of $f$, i.e.\ the image of the (finite) singular values of $f_0$ under $c$. Since we are interested only in the parameter space of $f_0$, without loss of generality we might assume that $\SV(f_0)\cap I(f_0)=\emptyset$, in particular, all exponentially bounded external addresses and potentials are realized. We say that a list of external addresses $\{\underline{s}_i\}_{i=1}^m$ and potentials $\{T_i\}_{i=1}^m$ \emph{admits finitely many non-trivial clusters}, if the union of the orbits of $f_0$ realizing $\{\underline{s}_i\}_{i=1}^m$ and $\{T_i\}_{i=1}^m$ contains finitely many non-trivial clusters. On the level of external addresses and potentials this means that there are only finitely many pairs of indexes $i,j$ and $k,l$ such that simultaneously hold $s_{ij}=s_{kl}$ and $F^j(T_i)=F^k(T_l)$.

It might be not immediately clear why in Theorem~\ref{thm:main_thm} we consider only the case of finitely many non-trivial clusters. A short answer is that if there are infinitely many of them, the construction of the invariant compact subset in Theorem~\ref{thm:invariant_subset} becomes more subtle; a particular example of what goes wrong is discussed right after it.

The goal of the next lemma is to separate $P_f$ and $S_0$ into two parts: one near $\infty$ with good estimates and ``straight'' dynamic rays, and one near the origin where the points are located somewhat more chaotically and ``entanglements'' of the dynamic rays might happen. This is a preparational statement \emph{solely} about $f_0$ and the standard spider.

\begin{lmm}[Good estimates for $S_0$ near $\infty$]
	\label{lmm:good_behavior_of_f_0}
	There exists $t'>0$ such that:
	\begin{enumerate}
		\item if $t_i>t_j>t'$, then $t_i-t_j>2$,
		\item if $t_{kl}>t_{ij}>t'$, then $\abs{a_{kl}}>\abs{a_{ij}}+2$,
		\item if $\rho\in \mathcal{P}'$ is bigger than $t'$, then all $a_{ij}$ with the potential less than $\rho$ are contained in $\mathbb{D}_{\rho-1}(0)$, while all $a_{ij}\in P_f$ with the potential bigger than $\rho$ are contained in the complement of $\mathbb{D}_{\rho+1}(0)$.
	\end{enumerate}	
\end{lmm}
\begin{proof}
	Follows immediately from the asymptotic formula~\ref{eqn:as_formula} and the fact that $s_{ij}/{\Re a_{ij}}\to 0$ as $j\to\infty$.
\end{proof} 

For all $\rho\in\mathcal{P}'$ we introduce the following notation. Let $D_\rho:=\mathbb{D}_\rho (0)$, and for $i\in\{1,...,m\}$ let $N_i=N_i(\rho)$ be the maximal $j$-index of the points $a_{ij}$ contained in $D_\rho$. For $\rho>t'$ the disk $D_\rho$ contains first $N_i+1$ points $\{a_{i0},...,a_{iN_i}\}$ of $\mathcal{O}_i$, whereas the other points of $\mathcal{O}_i$ are in $\mathbb{C}\setminus{D}_{\rho}$.

\subsubsection{Non-combinatorial estimates on $\hat{\varphi}(P_f)$}

The following proposition concerns the positions of the points in $P_f$ under the $\id$-type maps $\varphi$ and $\hat{\varphi}$ and under their isotopies $\varphi_u$ and $\hat{\varphi}_u$ to $c^{-1}$. In particular, if under $\varphi_u$ the points of $P_f$ outside of $D_\rho$ do not move ``much'', while the rest of the points of $P_f$ stays inside of $D_\rho$, then the first ones under $\hat{\varphi}_u$ do not move ``much'' as well, while those inside $D_\rho$ move inside of a left half-plane $\{z: \Re z<\rho/2\}$.

\begin{prp}[Non-combinatorial estimates on $\hat{\varphi}(P_f)$]
	\label{prp:good_big_disk_around_origin}
	There exists $k_1>t'$ such that for all $\rho\in\mathcal{P}'\cap [k_1,\infty]$ holds the following statement.
	
	Let $\varphi$ be an $\id$-type map with $\varphi(\SV(f))\subset D_\rho$, and let $\varphi_u,u\in[0,1]$ be an isotopy \idt\ maps such that $\varphi_0=c^{-1}$ and for all $u\in[0,1]$ $\varphi_u(\SV(f))\subset D_\rho$. Then: 
	\begin{enumerate}
		\item (\emph{Inside of $D_\rho$}) If $\varphi(a_{i(j+1)})\in D_\rho$, then $$\Re \hat{\varphi}(a_{ij})<\frac{\rho}{2}.$$
		\item (\emph{Crossing the boundary}) If $\left|\varphi(a_{i(N_i +1)})-a_{i(N_i +1)}\right|<1$, then $$\Re\hat{\varphi}\left(a_{iN_i}\right)< \frac{\rho}{2}.$$
		\item (\emph{Outside of $D_\rho$}) If $a_{ij}\notin{D}_\rho$ and $\left|\varphi_{u}(a_{i(j+1)})-a_{i(j+1)}\right|<1$ for all $u\in[0,1]$, then $$\left|\hat{\varphi}_u(a_{ij})-a_{ij}\right|<\frac{1}{j}.$$
	\end{enumerate}	
\end{prp}
\begin{remark}
	The choice of $\rho/2$ in (1)-(2) and $1/j$ in (3) as bounds is far from being rigid. Alternatively, we could have use, for example, $\varepsilon\rho$ where $0<\varepsilon<1$ in (1)-(2) and $1/2^j$ in (3).
\end{remark}
\begin{proof}	
	Let $g=\varphi\circ f\circ\hat{\varphi}^{-1}=q\circ \exp$ with $q(z)=z^d+b_{d-1}z^{d-1}+...+b_1 z+b_0$, and analogously $g_u=\varphi_u\circ f\circ\hat{\varphi}_u^{-1}=q_u\circ \exp$ with $q_u(z)=z^d+b_{d-1}^u z^{d-1}+...+b_1^u z+b_0^u$ and $g_0=f_0$.
	
	Restrict to $\rho>t'$.
	
	\textbf{(1)-(2)} From Lemma~\ref{lmm:complement_of_big_disk_under_preimage} we know that if $\rho$ is big enough and $SV(g)\subset D_\rho$, then $q^{-1}(\mathbb{D}_r(0))\subset\mathbb{D}_r(0)$ for $r\geq\rho$.
	
	Let $\abs{w}<\max_i \abs{a_{i(N_i+1)}}+1$ and $g(z)=w$. Then from the previous paragraph follows that $\abs{\exp z}=\abs{q^{-1}(w)}<\max_i \abs{a_{i(N_i+1)}}+1$. If $F(t)$ is the potential of $a_{i{(N_i+1)}}$ for which the maximum is realized, then we have
	
	$$\Re z=\log\abs{q^{-1}(w)}<\log\abs{w}<\log\left(\max_i \abs{a_{i(N_i+1)}}+1\right)<$$
	$$\log\left(\left(F(t)+2\pi \abs{s_{i(N_i+1)}}+1\right) +1\right)=\log\left(e^{dt}+2\pi \abs{s_{i(N_i+1)}}+1\right)=$$
	$$dt+\log\left(1+\frac{2\pi\abs{s_{i(N_i+1)}}+1}{F(t)+1}\right).$$
	
	Since the external address is exponentially bounded, we may assume that the absolute value of the second summand is less that $1$. At the same time $\rho$ is chosen equal to $\frac{t_k+t_{k+1}}{2}$, where $t_k,t_{k+1}$ are the consecutive elements of $\mathcal{P}$ and $t_{k+1}>t$. This means that $t\leq t_k$. If needed, increase $\rho$ so that $t_k>4$ and $t_k/t_{k+1}<1/4d$. Thus $$dt\leq dt_k=\frac{t_k+4dt_k -t_k}{4}<\frac{t_k+t_{k+1}-4}{4}<\frac{\rho}{2}-1.$$ Hence $\Re z<\rho/2$ if $\rho$ is sufficiently big.
	
	(3) First, since $\varphi_0=c^{-1}$, we have $\hat{\varphi}_0=\id$. Hence $\hat{\varphi}_0(a_{ij})=a_{ij}$ for all $a_{ij}\in P_f$.
	
	Consider some $a_{ij}\notin D_{\rho}$. We have the equality
	
	$$q_u\circ\exp\left(\hat{\varphi}_u(a_{ij})\right)=\varphi_{u}(a_{i(j+1)}).$$
	
	From the asymptotic formula~\ref{eqn:as_formula} we know that for $\rho$ big enough we have $\abs{a_{ij}}^{2d+1}+1\leq\abs{a_{i(j+1)}}$, and hence $\rho^{2d+1}+1<\abs{a_{i(j+1)}}$. Since $\abs{\varphi_{u}(a_{i(j+1)})-a_{i(j+1)}}<1$, we have $\rho^{2d+1}<\abs{\varphi_{u}(a_{i(j+1)})}$, and from Lemma~\ref{lmm:complement_of_big_disk_under_preimage} we obtain
	
	$$\left|\exp(\hat{\varphi}_u(a_{ij}))\right|^{1/2}>\rho.$$
	
	Now, let $$\varphi_{u}(a_{i(j+1)})-a_{i(j+1)}=\delta_u$$ with $\abs{\delta_u}<1$, or equivalently $$g_u\left(\hat{\varphi}_u(a_{ij})\right)=g_0\left(\hat{\varphi}_0(a_{ij})\right)+\delta_u=f(a_{ij})+\delta_u.$$
	
	From Lemma~\ref{lmm:bounds_of_coefficients_through_SV} we know that the coefficients of $q_u$ satisfy $\abs{b_k^u}<L\rho^\frac{d-k}{d}$. Hence
	
	$$g_u(\hat{\varphi}_u(a_{ij}))=f(a_{ij})+\delta_u \iff$$
	$$e^{d\hat{\varphi}_u(a_{ij})}\left(1+\frac{b_{d-1}^u}{\exp(\hat{\varphi}_u(a_{ij}))}+...+\frac{b_0^u}{\exp(d\hat{\varphi}_u(a_{ij}))}\right)=$$
	$$e^{da_{ij}}\left(1+\frac{b_{d-1}^0}{\exp(a_{ij})}+...+\frac{b_0^0}{\exp(da_{ij})}+\frac{\delta_u}{\exp(da_{ij})}\right)$$	
	$$\iff$$
	$$e^{d\hat{\varphi}_u(a_{ij})}\left(1+O\left(\frac{1}{\abs{\exp(\hat{\varphi}_u(a_{ij}))}^{1/2}}\right)\right)=e^{da_{ij}}\left(1+O\left(\frac{1}{\abs{\exp(a_{ij})}^{1/2}}\right)\right).$$
	
	Hence, after taking logarithm of both sides, we get
	$$\hat{\varphi}_u(a_{ij})-a_{ij}-\frac{2\pi i n_{ij}^u}{d}=O\left(\frac{1}{\abs{\exp(\hat{\varphi}_u(a_{ij}))}^{1/2}}\right)+O\left(\frac{1}{\abs{\exp(a_{ij})}^{1/2}}\right),$$
	where $n_{ij}^u\in\mathbb{Z}$. Furthermore, since $\abs{\exp(\hat{\varphi}_u(a_{ij}))}^{1/2}>\rho$, by increasing $\rho$ we can make the right hand side of the last expression arbitrarily small, in particular to make its absolute value less than $\pi i/d$ for all $u$. But then, since $\hat{\varphi}_0(a_{ij})=a_{ij}$, from the continuity of $\hat{\varphi}_u$ follows that $n_{ij}^u=0$ for all $j>N_i$. In particular, for all $\rho$ big enough holds
	$$\abs{\hat{\varphi}_u(a_{ij})-a_{ij}}<1.$$
	But then
	$$O\left(\frac{1}{\abs{\exp(\hat{\varphi}_u(a_{ij}))}^{1/2}}\right)=O\left(\frac{1}{\abs{\exp(a_{ij})}^{1/2}}\right)$$
	and
	$$\hat{\varphi}_u(a_{ij})-a_{ij}=O\left(\frac{1}{\abs{\exp(a_{ij})}^{1/2}}\right)=O\left(\exp\left(-\frac{\Re a_{ij}}{2}\right)\right)=O\left(\exp\left(-\frac{t_{ij}}{2}\right)\right).$$
	
	The expression on the right tends to $0$ much faster than the sequence $\{1/j\}_{j=1}^\infty$. So, after possibly increasing $\rho$, for $a_{ij}\notin D_\rho$ we have 
	
	$$\abs{\hat{\varphi}_u(a_{ij})-a_{ij}}<\frac{1}{j}.$$			
\end{proof}

Note that in the proof of Proposition~\ref{prp:good_big_disk_around_origin} (1)-(2) we showed even a bit more.

\begin{lmm}[Preimages of big disks]
	\label{lmm:preimages_of_big_disks}
	Let $\rho=(t_n+t_{n+1})/2\in\mathcal{P}'\cap [k_1,\infty]$, and let $g\in\mathcal{N}_d$ be such that $\SV(g)\subset D_\rho$.
	
	Then if $\abs{g(z)}<\max_i\abs{a_{i(N_i +1)}}+1$, we have
	$$\Re z< (d+1)t_n.$$
\end{lmm}

\subsubsection{No combinatorics near $\infty$ (almost)}

Proposition~\ref{prp:good_big_disk_around_origin} is supposed to estimate the positions of the marked points $\hat{\varphi}_u(P_f)$ when we know those of $\varphi_u(P_f)$. In particular, we obtain good enough estimates for the points $\hat{\varphi}_u(a_{ij})$ when $a_{ij}$ are outside of $D_\rho$. Nevertheless, when $a_{ij}\in D_\rho$, we only get an upper bound for the real parts of $\hat{\varphi}_u(a_{ij})$ though for the sake of invariance we want $\hat{\varphi}_u(a_{ij})\in D_\rho$. Hence we are also interested in bounds for the imaginary part and a lower bound for the real part. Proposition~\ref{prp:constant_homotopy_near_infinity} together with Theorem~\ref{thm:homotopy_type_under_pullback} will give an estimate on ``how many times'' $\varphi_u(R_{ij})$ ``wraps'' around $\varphi_u(\mathcal{O}_{ij})$ whence we get bounds for the imaginary part of $\hat{\varphi}_u(a_{ij})$. More precisely, Proposition~\ref{prp:constant_homotopy_near_infinity} says that $\varphi_u(R_{ij})$ generate a bounded amount of ``twists'' if under $\varphi_u$ the points $a_{ij}$ do not move much and do not rotate around each other.

Now, we want to prove that if the marked points near $\infty$ do not move ``much'' under $\varphi_u$, then the corresponding images under $\varphi_u$ of the ray tails $R_{ij}$ with big indices have uniformly bounded homotopy types.

\begin{prp}[No combinatorics near $\infty$]
	\label{prp:constant_homotopy_near_infinity}
	If $P_f$ has finitely many non-trivial clusters, then there exist constants $k_2>t'$ and $C>0$ such that the following statement holds.
	
	If $\rho\in\mathcal{P}'\cap[k_2,\infty]$, and $\varphi_u, u\in[0,1]$ is an isotopy of $\id$-type maps satisfying:
	\begin{enumerate}
		\item $\varphi_0|_{\mathbb{C}\setminus\mathbb{D}_\rho}=\id$,		
		\item $\varphi_u(a_{ij})\in D_\rho$ for $j\leq N_i$,
		\item $\abs{\varphi_u(a_{ij})-a_{ij}}<1/j$ for $j>N_i$,
	\end{enumerate}
	then for every $u$ and every $j>N_i$ we have $$\abs{W_{ij}^{\varphi_u}}<C.$$
\end{prp}
\begin{proof}
	
	First, note that if $\rho$ is big enough, then the points $a_{ij}\notin D_\rho$ move under $\varphi_u$ inside of the mutually disjoint disks $\mathbb{D}_{1/j}(a_{ij})$, and the distance between any two such disks is bounded from below by $\pi/2d$. 
	
	Since for all big enough $j$ the ray tail $R_{ij}=\varphi_0(R_{ij})$ has a strictly increasing real part (Lemma~\ref{lmm:monotonicity_of_rays}), for every $u\in[0,1]$ in the homotopy class (relative $P_f$) of $\varphi_u(R_{ij})$ there is a representative with the strictly increasing real part.
	
	Further, from the asymptotic formula~\ref{eqn:as_formula} follows that there is a universal constant $M>0$ such that for every $a_{ij}\in P_f$, every isotopy $\varphi_u$ satisfying conditions of the proposition, and every $u\in[0,1]$ there is at most $M$ points $a_{kl}\in \mathcal{O}_{ij}$ with $\Re\varphi_u(a_{kl})>\Re\varphi_u(a_{ij})-1$. Hence, from \cite[Lemma~6.6]{IDTT1} and \cite[Lemma~6.8]{IDTT1} we see that there exists a universal constant $C$ such that for all $\rho$ big enough and all $j>N_i$ we have
	$\abs{W_{ij}^{\varphi_u}}<C$.	
\end{proof} 

\subsubsection{Separation of preimages in $D_\rho$}

Next definition and Proposition~\ref{prp:division_over_derivative} are suited to take care of two subjects. First, they help to control the distance between points $\hat{\varphi}_u(a_{ij})$ inside of $D_\rho$. Hence this ``blocks'' one way of escape to $\infty$ in the \tei\ space. Second, since one controls the distance between $\varphi_u(a_{ij})$ and the asymptotic value $\varphi_u(a_{10})$, one obtains a lower bound for $\Re\hat{\varphi}_u(a_{ij})$.

\begin{defn}[Maximum of the derivative]
	For $\rho=\frac{t_n+t_{n+1}}{2}\in\mathcal{P}'$, define $$M_\rho:=\sup\limits_{\substack{g\in\mathcal{N}_d \\ SV(g)\subset D_\rho}}\sup\limits_{\Re z<{(d+1) t_n}}\abs{g'(z)}.$$	
\end{defn}

Note that from Lemma~\ref{lmm:bounds_of_coefficients_through_SV} follows immediately that
\begin{equation}
	\label{eqn:M_rho}
	M_\rho<Ke^{d^3 t_n}
\end{equation}
for some $K>0$ depending only on $d$. Also note that for $\rho$ big enough and every fixed function $g=p\circ\exp\in\mathcal{N}_d$, from the representation $g'(z)=p'\circ\exp(z)\exp(z)$ of its derivative we have
$$\sup\limits_{\Re z<{(d+1) t_n}}\abs{\exp(z)}<M_\rho,$$
and
$$\sup\limits_{\abs{z}<\exp((d+1) t_n)}\abs{p'(z)}<M_\rho.$$

Next proposition is needed to estimate mutual distances between marked points in $D_\rho$. 

\begin{prp}[Separation of preimages]
	\label{prp:division_over_derivative}	
	Let $\rho\in\mathcal{P}'\cap[k_1,\infty]$, and $\varphi$ be an $\id$-type map such that $\varphi(\SV(f))\subset D_\rho$. If $x,y\in\mathbb{C}\setminus\SV(f)$ are such that $$\varphi\circ f(x),\varphi\circ f(y)\in\overline{\mathbb{D}}_{\max\{|a_{i(N_i+1)}|+1\}}(0),$$ then
	$$\abs{\hat{\varphi}(x)-\hat{\varphi}(y)}\geq\frac{\abs{\varphi\circ f(x)-\varphi\circ f(y)}}{M_\rho}.$$
\end{prp}
\begin{proof}
	Let $\rho=\frac{t_n+t_{n+1}}{2}\in\mathcal{P}'$, and assume that $a_{k{(N_k+1)}}$ has the maximal potential $t$ among points $\{a_{i(N_i+1)}\}_{i=1}^m$.
	
	Let $g=\varphi\circ f\circ\hat{\varphi}^{-1}$. From Lemma~\ref{lmm:preimages_of_big_disks} we know that $\Re g^{-1}(\varphi\circ f(x))\leq{(d+1) t_n}$ if $\varphi\circ f(x)\in\overline{\mathbb{D}}_{|a_{k(N_k+1)}|+1}(0)$.
	
	Let $\gamma$ be the straight path joining $\hat{\varphi}(x)=g^{-1}(\varphi\circ f(x))$ to $\hat{\varphi}(y)=g^{-1}(\varphi\circ f(y))$. Then
	$$\abs{\varphi\circ f(x)-\varphi\circ f(y)}=\abs{\int_\gamma g'(z)dz}\leq M_\rho \abs{\hat{\varphi}(x)-\hat{\varphi}(y)}.$$	
\end{proof}

\subsection{Compact invariant subset}

We are finally ready to construct the invariant compact subset $\mathcal{C}_f$. In the first theorem we present the construction and prove the invariance. The statement that $\mathcal{C}_f$ is compact will be the content of the second theorem afterwards. 

\begin{thm}[Invariant subset]
	\label{thm:invariant_subset}
	Let $f=c\circ f_0$ be the quasiregular function defined in Section~\ref{sec:setup_and_contraction} so that $P_f$ contains only finitely many non-trivial clusters. Further, let $\rho\in\mathcal{P}'$ and $\mathcal{C}_f(\rho)\subset\mathcal{T}_f$ be the \emph{closure} of the set of points in the \tei\ space $\mathcal{T}_f$ represented by $\id$-type maps $\varphi$ for which there exists an isotopy $\varphi_u, u\in[0,1]$ \idt\ maps such that $\varphi_0=\id$, $\varphi_1=\varphi$, and the following conditions are simultaneously satisfied.
	\begin{enumerate}		
		\item (Marked points stay inside of $D_\rho$) If $j\leq N_i$,
		$$\varphi_u(a_{ij})\in D_\rho.$$
		\item (Precise asymptotics outside of $D_\rho$) If $j>N_i$, then
		$$\abs{\varphi_u(a_{ij})-a_{ij}}<1/j.$$		
		\item (Separation inside of $D_\rho$) If $j\leq N_i,l\leq N_k,ij\neq kl$, and $n=\min\{N_i+1-j, N_k+1-l\}$, then $$\abs{\varphi_u(a_{kl})-\varphi_u(a_{ij})}>\frac{\pi}{2d(M_\rho)^n}.$$	
		\item (Bounded homotopy) If $j\leq N_i$, then 		
		$$\abs{W_{ij}^{\varphi_u}}<A^{N_i+1-j}\left(\frac{(N_i+1)!}{j!}\right)^4 C$$		
		where $A$ and $C$ are the constants from Theorem~\ref{thm:homotopy_type_under_pullback} and Proposition~\ref{prp:constant_homotopy_near_infinity}, respectively.\\
	\end{enumerate}
	
	Then if $\rho\in\mathcal{P}'$ is big enough, the subset $\mathcal{C}_f(\rho)$ is well-defined, invariant under the $\sigma$-map and contains $[\id]$.	
\end{thm}

Let us give an overview of conditions (1)-(4) before the proof.

Conditions (1)-(2) say that the maps $\varphi$ have to be ``uniformly \idt'', that is, the marked points outside of a disk $D_\rho$ have precise asymptotics, while inside of $D_\rho$ we allow some more freedom.

Condition (3) tells us that the points inside of $D_\rho$ cannot come very close to each other --- it is necessary for keeping our set bounded (in the \tei\ metric). Moreover, it is needed to control the distance to the asymptotic value --- if a marked point is too close to it, then after the Thurston iteration its preimage has its real part close to $-\infty$, and this spoils condition (1).

Condition (4) takes care of the homotopy information and provides bounds for the leg homotopy words of legs with endpoints inside of $D_\rho$. Note that the analogous bounds for marked points outside of $D_\rho$ are encoded \emph{implicitly} in conditions (2) and (3).

\begin{proof}[Proof of Theorem~\ref{thm:invariant_subset}]
	After all preparations made earlier in this section, the proof is similar to the exponential case \cite[Theorem~7.1]{IDTT1}. Let $\mathcal{C}_f^\circ(\rho)\subset\mathcal{C}_f(\rho)$ be the set of points in $\mathcal{T}_f$ of which we take the closure in the statement of the theorem, that is, represented by $\id$-type maps $\varphi$ for which there exists an isotopy $\varphi_u, u\in[0,1]$ \idt\ maps such that $\varphi_0=\id$, $\varphi_1=\varphi$ and the conditions (1)-(4) are simultaneously satisfied. Since the $\sigma$-map is continuous, it is enough to prove invariance of $\mathcal{C}_f^\circ(\rho)$.
	
	First, note that for big $\rho$ the set $\mathcal{C}_f^\circ(\rho)$ contains $[c^{-1}]$: $c^{-1}$ can be joined to identity via the isotopy $c_u^{-1}$ where $c_u$ was constructed in Section~\ref{sec:setup_and_contraction}, and this isotopy satisfies the conditions (1)-(4).
	
	In view of Lemma~\ref{lmm:good_behavior_of_f_0} for $\rho\in\mathcal{P}'$ big enough the first $N_i+1$ point on each orbit $\mathcal{O}_i$ are contained in $D_\rho$, while the other points are outside. Moreover, due to the asymptotic formula~\ref{eqn:as_formula} those marked points from $D_\rho\cap P_f$ move under isotopy $\varphi_u$ only inside of $D_\rho$, while for every $a_{ij}\notin D_\rho$ the point $a_{ij}$ moves inside of a disk $D_{ij}$ of radius $1/j$, and all these disks $D_\rho$ and $\{D_{ij}\}$ are mutually disjoint and have mutual distances from each other bigger than $\pi/2d$.
	
	For all $\varphi\in\mathcal{C}_f^\circ(\rho)$ after concatenation with $c_u^{-1}$ we obtain the isotopy $\psi_u$ \idt\ maps with $\psi_0=c^{-1},\psi_1=\varphi$ and satisfying conditions (1)-(4). Note that then $\hat{\psi}_u$ is an isotopy \idt\ maps with $\hat{\psi}_0=\id$. Let $g_u(z)=\psi_u\circ f\circ\hat{\psi}_u^{-1}(z)=p_u\circ\exp(z)$. We want to prove that $\hat{\psi}_u$ satisfies each of the items (1)-(4): from this would follow that $\hat{\varphi}\in\mathcal{C}_f^\circ(\rho)$. 
	
	We prove that each of the conditions (1)-(4) for $\hat{\psi}$ follows from the conditions (1)-(4) for $\psi$. We assume that $\rho>\max\{k_1,k_2\}$.
	
	\textbf{(4)} Note that since the conditions of Proposition~\ref{prp:constant_homotopy_near_infinity} hold, for $j>N_i$ and, in particular, for $j=N_i +1$ we have $$\abs{W_{ij}^{\psi_u}}<C.$$
	
	Hence from Theorem~\ref{thm:homotopy_type_under_pullback} applied to $j\leq N_i +1$ we get $$\abs{W_{ij}^{\hat{\psi}_u}}<A(j+1)^4\abs{W_{i(j+1)}^{\psi_u}}<A(j+1)^4 A^{N_i+1-(j+1)}\left(\frac{(N_i+1)!}{(j+1)!}\right)^4 C=$$
	$$A^{N_i+1-j}\left(\frac{(N_i+1)!}{j!}\right)^4 C.$$
	
	\textbf{(1)} Without loss of generality we assume that $a_{10}$ is the asymptotic value of $f$ (i.e.\ the image of the asymptotic value of $f_0$ under $c$). We want to obtain upper and lower bounds on the real and imaginary parts of $\hat{\psi}_u(a_{ij})$ for $a_{ij}\in D_\rho$, which will give the desired estimate on $\abs{\hat{\psi}_u(a_{ij})}$.
	
	Let $j\leq N_i+1$. Then from Proposition~\ref{prp:good_big_disk_around_origin} we know that $\Re\hat{\psi}_u(a_{ij})<\rho/2$.
	
	From the estimate (3) we see that if $j\leq N_i$, then
	$$\abs{\psi_u(a_{i(j+1)})-\psi_u(a_{10})}>\frac{\pi}{2d(M_\rho)^{\max_i N_i+1}},$$
	and 
	$$\abs{p_u^{-1}(\psi_u(a_{i(j+1)}))-p_u^{-1}(\psi_u(a_{10}))}=\abs{\exp(\hat{\psi}_u(a_{ij}))-0}>\frac{\pi}{2d(M_\rho)^{\max_i N_i+2}}.$$
	Hence,
	$$\Re\hat{\psi}_u(a_{ij})>\log\left(\frac{\pi}{2d(M_\rho)^{\max_i N_i+2}}\right).$$
	If $\rho$ is big, from inequality~\ref{eqn:M_rho} follows that the right hand side of the last expression is bigger than $-\rho/2$.
	This proves that for $j\leq N_i$ we have
	$$\Re\hat{\psi}_u(a_{ij})>-\frac{\rho}{2}.$$
	
	Since $\abs{W_{i({j+1})}^{\psi_u}}<A^{N_i-j}\left(\frac{(N_i+1)!}{(j+1)!}\right)^4 C$, a preimage $p_u^{-1}(\psi_u({R}_{i(j+1)}))$ of the leg $\psi_u({R}_{i(j+1)})$ makes less than $A^{N_i+1-j}\left(\frac{(N_i+1)!}{j!}\right)^4 C$ loops around the singular value. Hence the difference between $2\pi s_{ij}/d$ and the imaginary part of $\hat{\psi}_u(a_{ij})$ will be less than $2\pi\left(A^{N_i+1-j}\left(\frac{(N_i+1)!}{j!}\right)^4 C+1\right)$. By making $\rho$ big enough we may assume that for all $i,j$ such that $j\leq N_i$ we have 
	$$\frac{2\pi \abs{s_{ij}}}{d}+2\pi \left(A^{N_i+1-j}\left(\frac{(N_i+1)!}{j!}\right)^4 C+1\right)<\frac{\rho}{2}.$$
	Hence, $\abs{\Im\hat{\psi}_u(a_{ij})}<\rho/2$.
	
	Consequently $\abs{\hat{\psi}_u(a_{ij})}=\sqrt{\abs{\Re\hat{\psi}_u(a_{ij})}^2+\abs{\Im\hat{\psi}_u(a_{ij})}^2}<\rho/\sqrt{2}<\rho$.
	
	\textbf{(3)} Follows directly from Proposition~\ref{prp:division_over_derivative}.
	
	\textbf{(2)} Follows directly from Proposition~\ref{prp:good_big_disk_around_origin}.	
\end{proof}

\begin{remark}
	Note that every point of $\mathcal{C}_f$ is evidently asymptotically conformal.
\end{remark}

Now we are ready to explain why a similar construction, where marked points $a_{ij}$ with $j>N_i$ move under $\varphi_u$ inside of small disjoint disks around $a_{ij}$, is generally impossible when we have infinitely many non-trivial clusters. Consider the function $f$ which has three singular orbits $\{a_{ij}\}, i=\{1,2,3\}$ such that the starting points $a_{10},a_{20},a_{30}$ have the same potential. Moreover, assume that $a_{10}$ has external address $\underline{s}_1=\{0\,0\,0\,0\,0\,0\,0\,0...\}$, whereas $\underline{s}_2=\{1\,0\,1\,0\,1\,0\,1\,0...\}$ and $\underline{s}_3=\{0\,1\,0\,1\,0\,1\,0\,1...\}$. Then for every $j\geq 0$, the point $a_{1j}$ belongs to a cluster with either $a_{2j}$ or $a_{3j}$. For $j>N_1+1$, if $a_{i(j+1)}$ is moving under $\varphi_u$ in the disk around $a_{i(j+1)}$ with radius $r_{i(j+1)}<1$, the point $a_{ij}$ is moving under $\hat{\varphi}_u$ approximately in the disk around $a_{ij}$ with radius $r_{ij}=r_{i(j+1)}/F'(t_{ij})$ where $t_{ij}$ is the potential of $a_{ij}$ (see computations in Proposition~\ref{prp:good_big_disk_around_origin}). Similar considerations apply to $a_{i(j+1)},a_{i(j+2)},...$ Hence, the invariance of these disks under $\sigma$ would imply that $r_{ij}=0$, which is clearly impossible unless $f$ is already entire.

Next, we prove that $\mathcal{C}_f(\rho)$ is compact.

\begin{thm}[Compactness]
	\label{thm:compact_subset}
	$\mathcal{C}_f(\rho)$ is a compact subset of $\mathcal{T}_f$.
\end{thm}
\begin{proof}
	We will prove that the set $\mathcal{C}_f^\circ(\rho)$ from the proof of Theorem~\ref{thm:invariant_subset} is pre-compact, or equivalently, since $\mathcal{T}_f$ is a metric space in the \tei\ metric, that every sequence $\{[\varphi^n]\}\subset\mathcal{C}_f^\circ(\rho)$ has a subsequence that converges to a point $[\varphi]\in\mathcal{T}_f$.
	
	Since each $[\varphi^n]\in\mathcal{C}_f^\circ(\rho)$, we can assume that every $\varphi^n$ is \idt\ and $\varphi_u^n$ is an isotopy with $\varphi_0^n=\id, \varphi_1^n=\varphi^n$ satisfying conditions (1)-(4) of Theorem~\ref{thm:invariant_subset}.
	
	Recall from the proof of Theorem~\ref{thm:invariant_subset} that for all $\rho\in\mathcal{P}'$ big enough all marked points move under an isotopy $\varphi_u\in\mathcal{C}_f^\circ(\rho)$ inside of mutually disjoint disks $D_\rho$ and $\{D_{ij}\}$ with radii of $D_{ij}$ tending to $0$ as $j\to\infty$. Hence we can assume that when $u\in[0,1/2]$, we have $\varphi^n_u|_{D_\rho}=\id$ for every $n$, and when $u\in[1/2,1]$, we have $\varphi^n_u|_{\cup_{k={N+1}}^\infty D_k}=\id$ for every $n$, that is, first the marked points inside of $D_\rho$ do not move, and afterwards do not move the marked point outside of $D_\rho$. It is clear from the definition of $W_{ij}^\varphi$ (see also \cite{IDTT1}) that in this case for every $a_{ij}\in D_\rho$ $W_{ij}^\varphi$ is bounded by a constant not depending on a particular choice of $\varphi$.  Thus, we simply need to prove the theorem in two separate cases: when only the marked points outside of $D_\rho$ move and when only the marked points inside of $D_\rho$ move. 
	
	From this point, after relabeling of every index $i,j$ by $i+mj-1$ and the replacement of \cite[Theorem~6.28]{IDTT1} by Theorem~\ref{thm:W_define_teich_point}, the proof repeats in every single detail the proof of \cite[Theorem~7.3]{IDTT1}.
\end{proof}

\subsection{Proof of Classification Theorem}

We are finally ready to prove the Classification Theorem~\ref{thm:main_thm}.

\begin{proof}[Proof of Classification Theorem~\ref{thm:main_thm}]
	
	The proof is analogous to the one of \cite[Theorem~1.1]{IDTT1}. Without loss of generality we may assume that all singular values of $f_0$ do not escape (such functions exist in the parameter space of $f_0$). Then due to Theorem~\ref{thm:as_formula} in $I(f_0)$ there are points $\{a_{i0}\}_{i=1}^m$ so that they escape on rays with starting potentials $\{T_i\}_{i=1}^m$ and external addresses $\{\underline{s}_i=(s_{i0},s_{i1},...)\}$.
	
	Let $f=c\circ f_0$ be the captured exponential function constructed as in Section~\ref{sec:setup_and_contraction}: i.e.\ with the singular values $\{a_{i0}\}_{i=1}^m$ escaping as under $f_0$. From Theorems~\ref{thm:invariant_subset} and \ref{thm:compact_subset} we know that there is a non-empty compact and invariant under $\sigma$ subset $\mathcal{C}_f=\mathcal{C}_f(\rho)\subset\mathcal{T}_f$ such that all its elements are asymptotically conformal. Next, from Theorem~\ref{thm:sigma_strictly_contracting} follows that $\sigma$ is strictly contracting in the \tei\ metric on $\mathcal{C}_f$.
	
	Thus, we have a strictly contracting map $\sigma$ on a compact complete metric space $\mathcal{C}_f$. It follows that $\mathcal{C}_f$ contains a fixed point $[\varphi_0]$. For the details that $\varphi_0\circ f\circ\hat{\varphi}_0^{-1}$ is the entire required function and it is unique we send the reader to the proof of \cite[Theorem~1.1]{IDTT1}. The proof is identical except that we have more than one singular orbit (modifications are trivial).
\end{proof}

\section{Appendix A}
\label{sec:app_A}

In this appendix we prove some properties of polynomials and of their compositions with the exponential, mainly those connecting the magnitude of the set of singular (or critical) values to the size of the coefficients. 

\begin{thm}[Critical values bound critical points]
	\label{thm:crit_pts_bounded}
	Fix some integer $d>1$. There exists a universal constant $M>0$ such that if $p$ is a monic polynomial of degree $d$ satisfying $p(0)=0$, and the critical values of $p$ are in the disk $\mathbb{D}_\rho(0)$ for some $\rho>0$, then its critical points are in $\mathbb{D}_{M\sqrt[d]{\rho}}(0)$.
\end{thm}
\begin{proof}
	If $p$ has all its critical points equal to $0$, then $p(z)=z^d$.
	
	Assume that $p$ has all its critical values equal to $0$. Then each critical point $c_i$ is a root of $p$ of order $\ord(c_i)+1$ where $\ord(c_i)$ is the order of critical point (or equivalently order of the root $c_i$ of $p'$). Since the degrees of $p$ and $p'$ differ by $1$, it is possible only if there is only one critical point $c$. But then $p(z)=(z-c)^d - (-c)^d$. Since $p(c)=0$, it follows that $c=0$. So $p(z)=z^d$ is the only monic polynomial satisfying $p(0)=0$ with all critical values equal to $0$.
	
	From now assume that neither all critical points nor all critical values are zero. Let $p'_C(z)=d(z-c_1)\cdot\dots\cdot(z-c_{d-1})$, where $C=(c_1,\dots,c_{d-1})\in\mathbb{C}^{d-1}$, and let $Q(z,C):=p_C(z)$. Then $Q:\mathbb{C}^d\to\mathbb{C}$ is a homogeneous polynomial of degree $d$ in variables $z, c_1,\dots,c_{d-1}$.
	
	For every $C\in \mathbb{C}^{d-1}\setminus\{0\}$ let $k_C:=\frac{\max\limits_i \abs{p_C(c_i)}}{\max\limits_i \abs{c_i}^d}$. From the considerations above follows that $k_C>0$ for every $C\in \mathbb{C}^{d-1}\setminus\{0\}$. On the other hand $k_C$ is a homogeneous function, hence it can be viewed as a continuous function on the complex projective space $\mathbb{CP}^{d-2}$. Due to compactness of $\mathbb{CP}^{d-2}$,  $k_C$ has a lower bound which has to be positive. Together with the degenerate case $p(z)=z^d$ the claim of the theorem follows.
\end{proof}

The following lemma provides estimates on the coefficients of a polynomial $p$ if we know the magnitude of the singular values of $p\circ\exp$.

\begin{lmm}[Singular values bound coefficients]
	\label{lmm:bounds_of_coefficients_through_SV} Fix some integer $d>0$.
	There exists a universal constant $L>0$ such that if $g=p\circ\exp$ where $p(z)=z^d+b_{d-1}z^{d-1}+...+b_1 z+b_0$ and singular values of $g$ are contained in the disk $\mathbb{D}_\rho(0)$ with $\rho$ big enough, then $\abs{b_k}<L\rho^\frac{d-k}{d}$.
\end{lmm}
\begin{proof}
	If $d=1$, then $g(z)=\exp(z)+b_0$, $g$ has the only singular value $b_0$ and the statement of the lemma is trivial. Hence assume that $d>1$.
	
	Represent $g$ in the form $g=q\circ E_\kappa$, where $E_\kappa(z)=\exp z +\kappa$ and $q(z)=z^d+a_{d-1}z^{d-1}+...+a_{1}z$. Such representation is generally not unique so just choose one of them. Then, in particular, $p(z)=q(z+\kappa)$. 
	
	Since $g'(z)=p'(\exp z)\exp z$, the critical values of the polynomial $p$ are the critical values of $g$, possibly together with $p(0)$ ($0$ is the omitted value of the exponential). Hence from Theorem~\ref{thm:crit_pts_bounded} we know that there is a constant $M_1>0$ such that all critical points of $q$ are contained in $\mathbb{D}_{M_1\sqrt[d]{\rho}}(0)$. From this, using the representation of the coefficients of $q'$ via critical points, we conclude that there is another constant $M_2>1$ such that $\abs{a_k}<M_2 \rho^\frac{d-k}{d}$.
	
	Now we need to estimate $\kappa$. The asymptotic value of $g$ is equal to $p(0)=q(\kappa)$ and $\abs{q(\kappa)}<\rho$. We want to prove that for $\rho$ big enough $\abs{\kappa}<2M_2\rho^\frac{1}{d}$. Indeed, otherwise there exist arbitrarily big $\rho$ such that $$\abs{q(\kappa)}=\abs{\kappa^d+a_{d-1}\kappa^{d-1}+...+a_{1}\kappa}=$$ $$\abs{\kappa}^d\abs{1+a_{d-1}/\kappa+...+a_1/(\kappa)^{d-1}}\geq\abs{\kappa}^d(1-\abs{a_{d-1}/\kappa}-...-\abs{a_1/(\kappa)^{d-1}})\geq$$
	$$\abs{\kappa}^d(1-\abs{\frac{M_2\rho^{1/d}}{2M_2\rho^{1/d}}}-...-\abs{\frac{M_2\rho^{(d-1)/d}}{(2M_2\rho^{1/d})^{d-1}}})>$$
	$$\abs{\kappa}^d(1-\frac{1}{2}-...-(\frac{1}{2})^{d-1})=\frac{\abs{\kappa}^d}{2^{d-1}}>2M_2^d\rho>\rho,$$
	which contradicts to the condition that $\abs{q(\kappa)}<\rho$. 
	
	Using the estimates on $\kappa$, $a_k$ and Newton's binomial formula we obtain the required estimate for $b_k$.
\end{proof}

Next statement describes the behavior of a polynomial $p$ outside of a disk containing the singular values of $p\circ\exp$ subject to the condition that this disk is big enough. 

\begin{lmm}[Preimages of outer disks]
	\label{lmm:complement_of_big_disk_under_preimage} Fix some integer $d>1$. If $\rho>1$ is big enough, then for every $g=p\circ\exp$ with monic polynomial $p$ of degree $d$ and singular values contained in $\mathbb{D}_\rho(0)$ we have:
	\begin{enumerate}
		\item $p^{-1}(\mathbb{D}_r(0))\subset\mathbb{D}_r(0)$ if $r\geq\rho$;
		\item $p^{-1}(\mathbb{D}_r(0))\supset\mathbb{D}_{\rho^2}(0)$ if $r\geq\rho^{2d+1}$.
	\end{enumerate}
\end{lmm}
\begin{proof}	
	Let $p(z)=z^d+b_{d-1}z^{d-1}+...+b_1 z+b_0$.	
	From Lemma~\ref{lmm:bounds_of_coefficients_through_SV} we know that if singular values of $g$ are contained in $\mathbb{D}_\rho(0)$, then $\abs{b_k}<L\rho^\frac{d-k}{d}$.
	
	(1) Let $p(z)=w$ where $\abs{w}>\rho$. Assume that $\abs{z}\geq\abs{w}>\rho$. Then we have $$\abs{w}=\abs{p(z)}=\abs{z^d+b_{d-1}z^{d-1}+...+b_1 z+b_0}=$$
	$$\abs{z}^d\left|1+\frac{b_{d-1}}{z}+...+\frac{b_1}{z^{d-1}}+\frac{b_0}{z^d}\right|\geq$$
	$$\abs{z}^d\left(1-\left|\frac{b_{d-1}}{z}\right|-...-\left|\frac{b_1}{z^{d-1}}\right|-\left|\frac{b_0}{z^d}\right|\right)>$$
	$$\abs{w}^d\left(1-\left|\frac{L\rho^{1/d}}{\rho}\right|-...-\left|\frac{L\rho^{(d-1)/d}}{\rho^{d-1}}\right|-\left|\frac{L\rho}{\rho^d}\right|\right)>\abs{w}$$
	if $\rho$ is sufficiently big.
	
	(2) We want to prove that $p(\mathbb{D}_{\rho^2}(0))\subset\mathbb{D}_{\rho^{2d+1}}(0)$ if $\rho$ is sufficiently big. Let $\abs{z}<\rho^2$. Then 
	
	$$\abs{p(z)}=\abs{z^d+b_{d-1}z^{d-1}+...+b_1 z+b_0}\leq$$
	$$\abs{z^d}+\abs{b_{d-1}z^{d-1}}+...+\abs{b_1 z}+\abs{b_0}\leq$$
	$$\rho^{2d}+L\rho^{1/d}\rho^{2(d-1)}+...+L\rho^{(d-1)/d}\rho^2+L\rho<(dL+1)\rho^{2d}<\rho^{2d+1}$$
	if $\rho$ is big enough.
\end{proof}

\section{Acknowledgements}

I would like to express our gratitude to our research team in Aix-Marseille Universit\'e, especially to Dierk Schleicher who supported this project from the very beginning, Sergey Shemyakov who carefully proofread all drafts, as well as to  Kostiantyn Drach, Mikhail Hlushchanka, Bernhard Reinke and Roman Chernov for uncountably many enjoyable and enlightening discussions of this project at different stages. I also want to thank Dzmitry Dudko for his multiple suggestions that helped to advance the project, Lasse Rempe for his long list of comments and relevant questions, and Adam Epstein for important discussions especially in the early stages of this project. 

Finally, I am grateful to funding by the Deutsche Forschungsgemeinschaft and the European Research Council with the Advanced Grant “Hologram” (695621) whose support provided excellent conditions for the development of this project.

\end{document}